\documentclass[11pt]{article}
\usepackage{tikz}
	\usetikzlibrary{math}
\usepackage{amsthm, amssymb, amsmath, cancel}
\usepackage[capitalise, noabbrev]{cleveref}
\usepackage{bold-extra}
\usepackage{xcolor}
\usepackage{appendix}
\usepackage{multirow}
\theoremstyle{plain}
	\newtheorem{theorem}{Theorem}
	
	\newtheorem{proposition}[theorem]{Proposition}
	\newtheorem{lemma}[theorem]{Lemma}
	\newtheorem{corollary}[theorem]{Corollary}

\theoremstyle{definition}
	\newtheorem{example}[theorem]{Example}
	\newtheorem{definition}[theorem]{Definition}
\theoremstyle{remark}

\DeclareMathOperator{\STS}{STS}
\DeclareMathOperator{\mex}{mex}
\newcommand{\Z}{\mathbb{Z}}
\newcommand{\F}{\mathbb{F}}
\newcommand{\nofil}{\textsc{Nofil }}
\newcommand{\nofilpunct}{\textsc{Nofil}}
\newcommand{\nodekayles}{\textsc{Node Kayles }}
\newcommand{\nodekaylespunct}{\textsc{Node Kayles}}

\usepackage{xspace}
\makeatletter
\DeclareRobustCommand\onedot{\futurelet\@let@token\@onedot}
\def\@onedot{\ifx\@let@token.\else.\null\fi\xspace}
 
\def\ie{{i.e}\onedot}

\def\etal{{et al}\onedot}
\makeatother

\begin{document}
\begin{center}
  \uppercase{\bf The combinatorial game nofil 
 played on Steiner Triple Systems}

\vskip 20pt
{\bf Melissa A. Huggan}\footnote{Supported by the Natural Sciences and Engineering Research Council of Canada (funding reference number PDF-532564-2019).}\\
{\small Department of Mathematics, Ryerson University\\Toronto, Ontario, Canada}\\
{\tt melissa.huggan@ryerson.ca}\\ 
\vskip 10pt
{\bf Svenja Huntemann}\footnote{Supported by the Natural Sciences and Engineering Research Council of Canada (funding reference number PDF-516619-2018).}\\
{\small Dept.\ of Math.\ and Phys.\ Sciences, Concordia University of Edmonton\\Edmonton, Alberta, Canada}\\
{\tt svenja.huntemann$@$concordia.ab.ca}\\ 
\vskip 10pt
{\bf Brett Stevens}\footnote{Supported by the Natural Sciences and Engineering Research Council of Canada (funding reference number RGPIN 06392).}\\
{\small School of Mathematics and Statistics, Carleton University\\Ottawa, Ontario, Canada}\\
{\tt brett@math.carleton.ca}\\ 
\end{center}
\begin{abstract}
We introduce an impartial combinatorial game on Steiner triple systems called \nofilpunct. Players move alternately, choosing points of the triple system. If a player is forced to fill a block on their turn, they lose. We explore the play of \nofil on all Steiner triple systems up to order 15 and a sampling for orders 19, 21, and 25.  We determine the optimal strategies by computing the nim-values for each game and its subgames.  The game \nofil can be thought of in terms of play on a corresponding hypergraph. As game play progresses, the hypergraph shrinks and will eventually be equivalent to playing the game \textsc{Node Kayles} on an isomorphic graph. \textsc{Node Kayles} is well studied and understood.  Motivated by this, we study which \textsc{Node Kayles} positions can be reached, i.e.\ embedded into a Steiner triple system. We prove necessary conditions and sufficient conditions for the existence of such graph embeddings and conclude that the complexity of determining the outcome of the game \nofil on Steiner triple systems is PSPACE-complete.
\end{abstract}


\section{Introduction}
Many games are played on graphs. Focusing on particular graph families allows for investigating how graph properties and structures give insight into player strategies and whether a player can win the game. Playing games on combinatorial designs has been explored in several contexts. For example, Cops and Robbers, a pursuit-evasion game, was played on designs  in~\cite{BB2012}, where it was shown that extremal cop-numbers are attained in certain families of designs. More recently, it was shown in \cite{BHM20} that for Cops and Robbers with an invisible robber, a strategy for capture can be derived from the  structure of the design. Tic-Tac-Toe on a finite affine or projective plane has been studied~\cite{CarrollDougherty} and the outcome determined for every finite plane.  The first player can win in the affine planes of orders 2, 3, and 4 and the projective plane of order 2. The second player can force a draw in all other finite planes. Eric Mendelsohn coined a triple packing game in the 1980s which belongs to the class of impartial combinatorial games.  The triple packing game was explored in~\cite{Huggan2015}, and although the comprehensive Sprague-Grundy Theory of impartial games applies, it is presumed to be difficult. 

In this paper, we examine a game played on a block design which we call \nofil ({\bf N}ext {\bf O}ne to {\bf F}ill {\bf I}s the {\bf L}oser). The game ruleset is as follows.
\begin{definition}[\nofilpunct]
The \emph{board} is a block design. On their turn, a player chooses an unplayed point from the design which does not complete a block of {\em played} points. Chosen points are simply considered \emph{played}, regardless of the player who chose them. If all but one point from any block have been played, then the last point is unplayable. If every unplayed point is also unplayable, then no move is possible, the game ends, and the last player to move wins. 
\end{definition}

The set of played points at the end of the game form a maximal independent set in the sense of not containing any hyperedge (block) \cite{MR309807}. We note that the game can be played on any hypergraph but we are interested in showing that the more restricted universe of designs contains all the complexity of a similar game played on graphs.  We will focus on studying \nofil specifically on Steiner triple systems.
\begin{definition}[\cite{Handbook}]
A \emph{Steiner triple system} ($\STS(v)$) ($V$, $\mathcal{B}$) is a set $V$ of $v$ elements (points or vertices) together with a collection $\mathcal{B}$ of $3$-subsets (blocks or triples) of $V$ with the property that every $2$-subset of $V$ occurs in exactly one block $B \in \mathcal{B}$. The size of $V$ is the \emph{order} of the $\STS$.
\end{definition}
For \nofil played on a Steiner triple system, if two out of three points of a block have been played, the last point is unplayable. 
At any time during the game, the vertices/points can be divided into three groups:
\begin{itemize}
\item \textbf{Played}: Those that have been chosen already. We denote this set by $P$.  The rules of the game imply that for every block $B$ in the design we have $B \not\subseteq P$.
\item \textbf{Unplayable}: If all but one point in a block have been played, the last point is unplayable.  We denote the set of such points by $U$. 
\item \textbf{Available}: The set of remaining unplayed points, which are still playable, is denoted $A$.
\end{itemize}

On each turn, the set of played points, $P$, increases by one and the remaining points can be partitioned between those that are unplayable, $U$, because they appear on a triple with two already played points, and those that are available, $A$, which have at most one played point in every block on which they appear. The set $A$ decreases by at least one each turn, while $U$ is non-decreasing.  If we concentrate on just the playable points, we can restrict our attention to the hypergraph on $A$, inherited from the $\STS$, which encodes the available plays and any restrictions among points. Its hyperedges are those maximal subsets of blocks which consist of available points and which, if all played, would complete a block of played points. To demonstrate what these maximal subsets are, based on the sets $P$, $A$ and $U$, we first analyze a sample game play on an $\STS(9)$ in \cref{ex: example}, then discuss them in general.  Whenever possible, we denote a triple containing $x,y, z \in V$ by $xyz$, otherwise we will use set notation, $\{x,y,z\}$.

\begin{example}\label{ex: example}
Consider an $\STS(9)$, shown in \cref{Fig: STS9}, where the points are $V = \{1, 2, 3, 4, 5, 6, 7, 8, 9\}$ and the blocks are
\[\mathcal{B}=\{123, 456, 789, 147, 258,  369, 159, 267, 348, 168, 249, 357\}.\] 

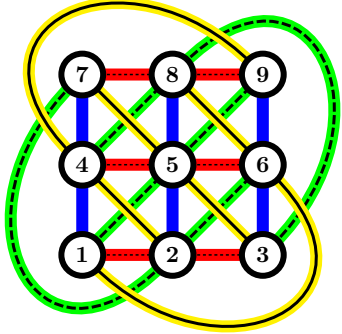
\begin{figure}[th!]
  \begin{center}
\scalebox{0.80}{
    \begin{tikzpicture}
      \tikzmath{\y1 = 0.5; \y2=1.5; }
      \node (001) at  (0*\y2,0*\y2) [circle,draw=black,line width=1.0mm] {};
      \node (101) at  (1*\y2,0*\y2) [circle,draw=black,line width=1.0mm] {};
      \node (201) at  (2*\y2,0*\y2) [circle,draw=black,line width=1.0mm] {};
      
      \node (011) at  (0*\y2,1*\y2) [circle,draw=black,line width=1.0mm] {};
      \node (111) at  (1*\y2,1*\y2) [circle,draw=black,line width=1.0mm] {};
      \node (211) at  (2*\y2,1*\y2) [circle,draw=black,line width=1.0mm] {};
      
      \node (021) at  (0*\y2,2*\y2) [circle,draw=black,line width=1.0mm] {};
      \node (121) at  (1*\y2,2*\y2) [circle,draw=black,line width=1.0mm] {};
      \node (221) at  (2*\y2,2*\y2) [circle,draw=black,line width=1.0mm] {};
\begin{scope}
      \draw [-,cap=rect,line width=2mm,draw=green] (001) -- (111);
      \draw [-,cap=rect,line width=0.5mm,draw=black, dashed] (001) -- (111);
      \draw [-,cap=rect,line width=2mm,draw=green] (111) -- (221);
      \draw [-,cap=rect,line width=0.5mm,draw=black, dashed] (111) -- (221);
      \draw [-,cap=rect,line width=2mm,draw=green] (121) to[out=45,in=45,distance=2*\y2 cm]  (201);
      \draw [-,cap=rect,line width=2mm,draw=green] (011) -- (121);
      \draw [-,cap=rect,line width=0.5mm,draw=black, dashed] (121) to[out=45,in=45,distance=2*\y2 cm]  (201);
      \draw [-,cap=rect,line width=0.5mm,draw=black, dashed] (011) -- (121);
      \draw [-,cap=rect,line width=2mm,draw=green] (021) to[out=225,in=225,distance=2*\y2 cm]  (101);
      \draw [-,cap=rect,line width=2mm,draw=green] (101) -- (211);
      \draw [-,cap=rect,line width=0.5mm,draw=black,dashed] (021) to[out=225,in=225,distance=2*\y2 cm]  (101);
      \draw [-,cap=rect,line width=0.5mm,draw=black,dashed] (101) -- (211);

      \draw [-,cap=rect,line width=2mm,draw=yellow] (021) -- (111);
      \draw [-,cap=rect,line width=0.5mm,draw=black] (021) -- (111);
      \draw [-,cap=rect,line width=2mm,draw=yellow] (111) -- (201);
      \draw [-,cap=rect,line width=0.5mm,draw=black] (111) -- (201);
      \draw [-,cap=rect,line width=2mm,draw=yellow] (221) to[out=135,in=135,distance=2*\y2 cm]  (011);
      \draw [-,cap=rect,line width=2mm,draw=yellow] (011) -- (101);
      \draw [-,cap=rect,line width=0.5mm,draw=black] (221) to[out=135,in=135,distance=2*\y2 cm]  (011);
      \draw [-,cap=rect,line width=0.5mm,draw=black] (011) -- (101);
      \draw [-,cap=rect,line width=2mm,draw=yellow] (001) to[out=315,in=315,distance=2*\y2 cm]  (211);
      \draw [-,cap=rect,line width=2mm,draw=yellow] (121) -- (211);
      \draw [-,cap=rect,line width=0.5mm,draw=black] (001) to[out=315,in=315,distance=2*\y2 cm]  (211);
      \draw [-,cap=rect,line width=0.5mm,draw=black] (121) -- (211);

      \draw [-,cap=rect,line width=2mm,draw=red] (011) -- (111);
      \draw [-,cap=rect,line width=0.25mm,draw=black, dotted] (011) -- (111);
      \draw [-,cap=rect,line width=2mm,draw=red] (111) -- (211);
      \draw [-,cap=rect,line width=0.25mm,draw=black, dotted] (111) -- (211);
      \draw [-,cap=rect,line width=2mm,draw=red] (001) -- (101);
      \draw [-,cap=rect,line width=0.25mm,draw=black, dotted] (001) -- (101);
      \draw [-,cap=rect,line width=2mm,draw=red] (101) -- (201);
      \draw [-,cap=rect,line width=0.25mm,draw=black, dotted] (101) -- (201);
      \draw [-,cap=rect,line width=2mm,draw=red] (021) -- (121);
      \draw [-,cap=rect,line width=0.25mm,draw=black, dotted] (021) -- (121);
      \draw [-,cap=rect,line width=2mm,draw=red] (121) -- (221);
      \draw [-,cap=rect,line width=0.25mm,draw=black, dotted] (121) -- (221);
      
      \draw [-,cap=rect,line width=2mm,draw=blue] (001) -- (011);
      \draw [-,cap=rect,line width=2mm,draw=blue] (011) -- (021);
      \draw [-,cap=rect,line width=2mm,draw=blue] (101) -- (111);
      \draw [-,cap=rect,line width=2mm,draw=blue] (111) -- (121);
      \draw [-,cap=rect,line width=2mm,draw=blue] (201) -- (211);
      \draw [-,cap=rect,line width=2mm,draw=blue] (211) -- (221);
\end{scope}

      \node (001) at  (0*\y2,0*\y2) [circle,draw=black,fill=white, line width=1.0mm] {\bf 1};
      \node (101) at  (1*\y2,0*\y2) [circle,draw=black,fill=white, line width=1.0mm] {\bf 2};
      \node (201) at  (2*\y2,0*\y2) [circle,draw=black,fill=white, line width=1.0mm] {\bf 3};
      
      \node (011) at  (0*\y2,1*\y2) [circle,draw=black,fill = white, line width=1.0mm] {\bf 4};
      \node (111) at  (1*\y2,1*\y2) [circle,draw=black, fill = white,line width=1.0mm] {\bf 5};
      \node (211) at  (2*\y2,1*\y2) [circle,draw=black, fill = white,line width=1.0mm] {\bf 6};
      
      \node (021) at  (0*\y2,2*\y2) [circle,draw=black,fill = white,line width=1.0mm] {\bf 7};
      \node (121) at  (1*\y2,2*\y2) [circle,draw=black,fill = white,line width=1.0mm] {\bf 8};
      \node (221) at  (2*\y2,2*\y2) [circle,draw=black,fill = white,line width=1.0mm] {\bf 9};
    \end{tikzpicture}}
\caption{A Steiner triple system of order $9$.}\label{Fig: STS9}
  \end{center}
 \end{figure}

A game on the $\STS(9)$ could play out as summarized in \cref{ex: steiner 9}, where the columns $P$, $A$, and $U$ give the respective disjoint sets of points described above. The \emph{blocks} column shows the list of blocks of the design, where points within a block are overlined if they have been played. No block is allowed to have all three of its points overlined. The \emph{available hypergraph} column describes the restrictions of points within blocks. For example, ``$23$'' indicates that points $2$ and $3$ are on a block together, and both cannot be chosen because the third point in the block has been played. Consider a block $348$ and suppose $3$ becomes unplayable. Then on the next turn, the $48$ hyperedge is removed because there is no restriction on playing both $4$ and $8$, as the block can never be completed. The last column, \emph{hypergraph figure}, presents the modification to \cref{Fig: STS9} throughout the proposed game play. The greyed out vertices represent unplayable points. The blacked out vertices represent the played points. All games played on the unique $\STS(9)$ are equivalent to this sequence of play, up to isomorphism. In this example, after turn 4 the game is over and Player 2 has won the game. 
\renewcommand{\arraystretch}{1.5}
\begin{table}[ht!]
\begin{center}
\begin{tabular}{|c|c|c|c|c|c|c|}\hline
\multirow{2}{*}{Turn}&\multirow{2}{*}{$P$}&\multirow{2}{*}{$A$}&\multirow{2}{*}{$U$}&\multirow{2}{*}{Blocks}&Available&Hypergraph\\
&&&&& Hypergraph&Figure\\\hline\hline
 \multirow{4}{*}{$1$}&\multirow{4}{*}{$1$}&&&$\overline{1}23, 456, 789,$ &$23, 456, 789,$&\multirow{4}{*}{\includegraphics[width=25mm, height=25mm]{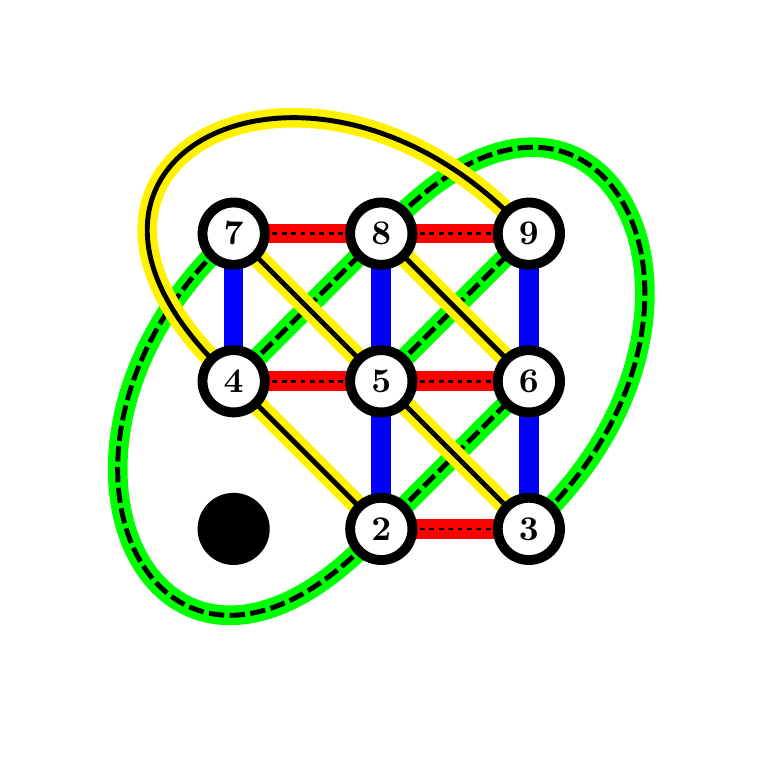}}
\\
&&$2, 3, 4, 5, $&&$ \overline{1}47, 258, 357,$&$ 47, 258, 357,$&\\
    &&$6, 7, 8, 9$&& $369, \overline{1}59, 267$&$369, 59, 267,$&\\
    &&&& $348, \overline{1}68, 249$&$ 348, 68, 249$&\\\hline
 \multirow{4}{*}{$2$}&\multirow{4}{*}{$1, 2$}&&\multirow{4}{*}{$3$}&$\overline{1}\overline{2}3, 456, 789, $ &$456, 789,$&\multirow{4}{*}{\includegraphics[width=25mm, height=25mm]{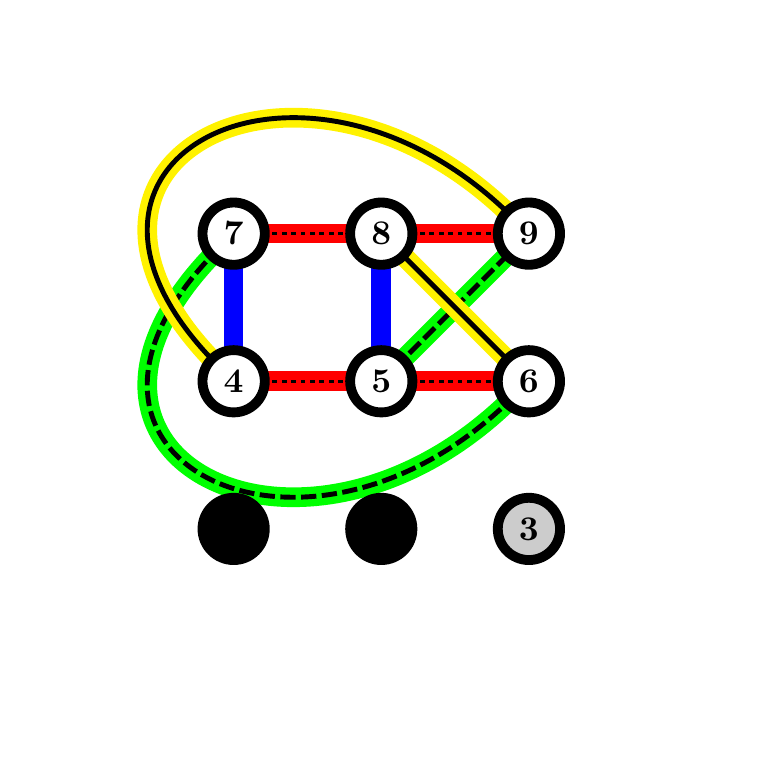}}\\
&&$4, 5, 6,$&&$\overline{1}47, \overline{2}58, 357,$& $ 47, 58,$&\\
    && 7, 8, 9&& $369, \overline{1}59, \overline{2}67,$&$59, 67,$& \\
&&&&$ 348, \overline{1}68, \overline{2}49$ &$ 68, 49$&\\\hline  
\multirow{4}{*}{$3$}&\multirow{4}{*}{$1, 2, 6$}&\multirow{4}{*}{$4, 5, 9$}&\multirow{4}{*}{$3, 7, 8$}&$\overline{1}\overline{2}3, 45\overline{6}, 789, $&\multirow{4}{*}{ $45, 59, 49$}&\multirow{4}{*}{\includegraphics[width=25mm, height=25mm]{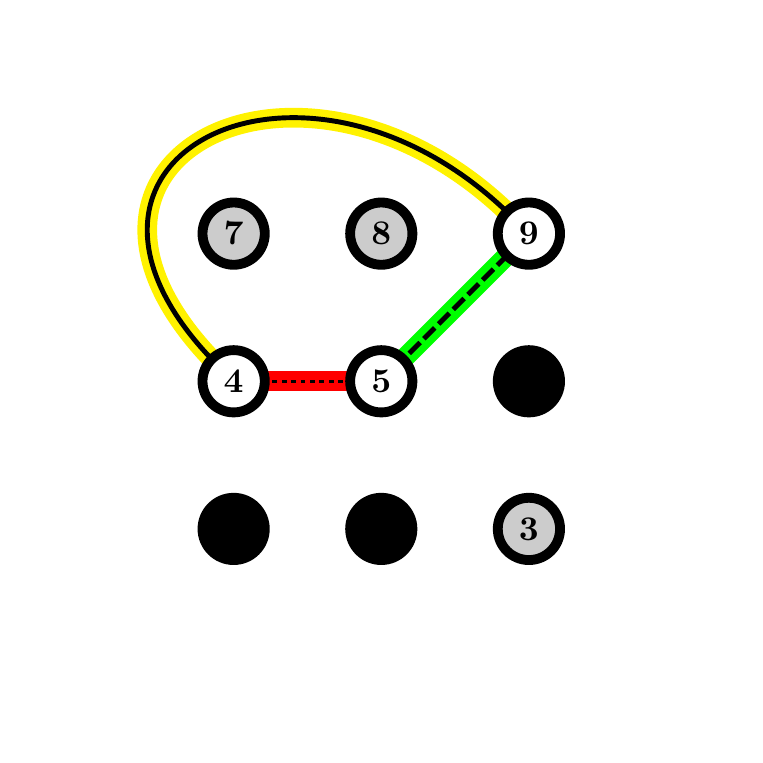}}\\
&&&&$\overline{1}47, \overline{2}58, 357,$ &&\\
    &&&& $3\overline{6}9, \overline{1}59, \overline{2}\overline{6}7,$&&\\
&&&&$ 348, \overline{1}\overline{6}8, \overline{2}49$&&\\\hline  
\multirow{4}{*}{$4$}&\multirow{4}{*}{$1, 2, 6, 4$}&&\multirow{4}{*}{$3, 7, 8, 5, 9$}&$\overline{1}\overline{2}3, \overline{4}5\overline{6}, 789,$ &&\multirow{4}{*}{\includegraphics[width=25mm, height=25mm]{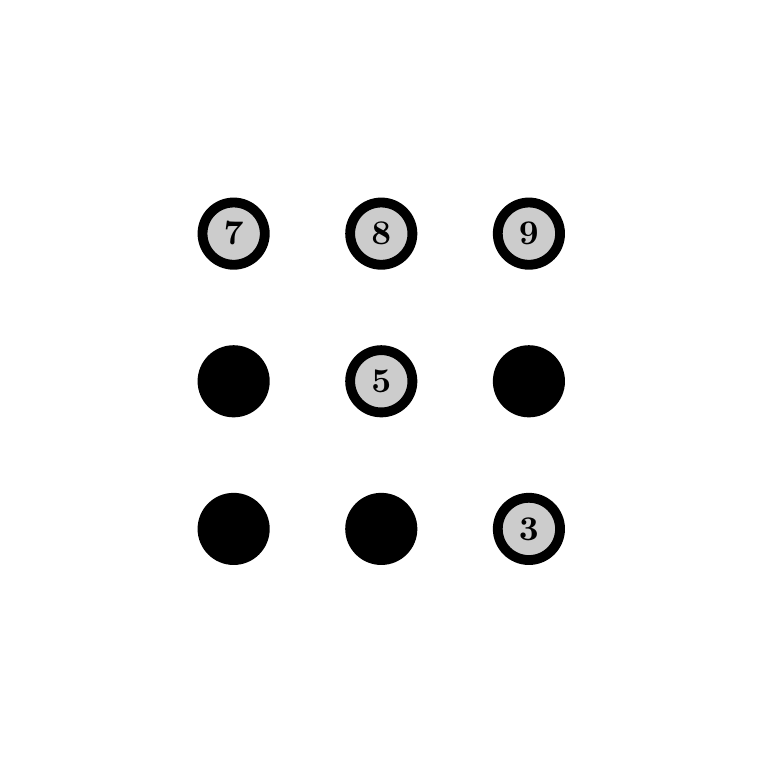}} \\
&&&&$ \overline{1}\overline{4}7, \overline{2}58, 357,$&&\\
&&&& $3\overline{6}9, \overline{1}59, \overline{2}\overline{6}7,$&& \\
&&&&$ 3\overline{4}8, \overline{1}\overline{6}8, \overline{2}\overline{4}9$&&\\\hline
\end{tabular}
\caption{Example of game play on an $\STS(9)$.}\label{ex: steiner 9}
\end{center}
\end{table}
\renewcommand{\arraystretch}{1}
\end{example}

When playing $\nofil$on a general $\STS(v)$, we first consider triples that contain unplayable points. If a triple $xyz$ has at least one unplayed point $x$ which is unplayable, then no restrictions are derived for the points $y$ and $z$ from that block. However, one or more of them could become unplayable based on restrictions caused by the play on other blocks. If there are no available points on the block, this is self evident. If there are available points on such a block, then since there is one unplayable point, no matter how play proceeds, this block can never possibly contain three played points.  Thus no subset of a block which contain at least one unplayable point is a hyperedge. 

If a block has no unplayable points, it can contain only played and available points.  Any two played points on a block determine that the third point on the block is unplayable. Thus blocks with two played points also do not contribute to the hyperedges.  The only remaining blocks to consider either contain three available points or two available points and a played point.  The former are included in their entirety as a hyperedge and the latter are included as a hyperedge of size two containing the two available points.  We call this the {\em available hypergraph}.  

Game theorists are interested in understanding the computational complexity of a game to get a measure of how difficult it will be to determine the winner. This is a decision problem as follows: \emph{Given a position $G$, can the first player win: yes or no?} Here, the position $G$ could be any possible position that can occur during game play or we may restrict to starting positions only. If the answer can be determined with a polynomial amount of storage (with respect to the size of the input information), then the problem is considered to be in PSPACE. PSPACE-hard problems are at least as hard as all the hardest problems in PSPACE. Note that there is no restriction on computation time, although a problem in P, which are those problems requiring a polynomial amount of time, only has time to access a polynomial amount of memory. A problem is PSPACE-complete if it is in PSPACE and is also PSPACE-hard. In other words, it is among the hardest problems in PSPACE. It is important to note that when studying the decision problem for starting positions only, the computational complexity might be ``simpler'' as optimal game play from a starting position could avoid the most difficult positions. For more background in complexity theory and how it relates to game theory, see~\cite{Demaine_09}.

Knowing that the game of interest, after a sequence of plays, is equivalent to another game for which we know the computational complexity tells us that the original game is at least as complex as the games embedded within it in this manner. When the available hypergraph is a graph, \nofil is equal to another game and we will focus on embedding it into \nofilpunct.  In the example, the available hypergraph is first isomorphic to a graph at the end of the third turn, where it is isomorphic to $K_3$.  Once the available hypergraph is a graph, it will continue to be a graph until the end of the game.

The combinatorial game played on an arbitrary graph, where players alternately choose a vertex that is not adjacent to a played vertex, and mark it as ``played'', is called \nodekaylespunct. An equivalent formulation of the game focuses on the available graph, which is the graph induced by the available points.  The second point on an edge with a played point is unavailable and thus the available graph is obtained by deleting a played vertex and all its neighbours.  \nodekayles has been well studied in combinatorial game theory. Determining the winning player in \nodekayles for a given graph is PSPACE-complete \cite{schaefer_78}. For many families of graphs, determining the winning player when starting on one of these graphs is a polynomial problem \cite{bodlaender_93,bodlaender_02,fleischer_06,bodlaender_11}. Note that the complexity for \nodekayles is the same when restricting to starting positions only as any position throughout the game is equivalent to a starting position by using the equivalent formulation.
 For more background on \nodekaylespunct, see \cite{sopena_09}, where it is shown to be an excellent venue to explore the differences between Conway's various combinatorial game theory winning conditions and ways of combining games. 

\nofil ends when $A = \emptyset$ and the available hypergraph is empty.  Since an empty hypergraph is also a graph, there is a turn in every game play when the available hypergraph on $A$ is simply a graph, possibly one with no edges or no vertices. In the above example, this first happens at turn $3$ where the available hypergraph is the graph $K_{3}$ with vertices labeled as $4$, $5$, and $9$. From this turn forward the players are effectively playing \nodekayles on $K_3$ and the available hypergraphs for all remaining turns will be graphs.  In the case of $K_3$, once one of the points is chosen, the game terminates. Different lines of play may leave different graphs on which \nodekayles is played.  For this reason, determining the winner of the game on the design, might require analyzing \nodekayles on the possible available hypergraphs and graphs that could appear.  In particular the complexity of determining the winner on a design or the available hypergraph is bounded below by the complexities for determining the winner on the positions which appear, including the positions equivalent to  \nodekaylespunct. We are interested in which graphs can appear as available hypergraphs.  We call these graphs {\em embedded} in the original design.  Understanding these substructures will give insight into the complexity of the original game. 

\subsection{Combinatorial Game Theory}\label{ss: background}
A \emph{combinatorial game} is a two player game of perfect information and no chance, where players move alternately.  A well-studied winning convention is \emph{normal play}, where the player without a move on their turn loses. Cops and Robbers and Tic-Tac-Toe are combinatorial games, but neither uses the normal play convention for winning. Mendelsohn's triple packing game is a combinatorial games with normal play, as is \nofilpunct.  

Often times we are interested in who wins a game and the winner's strategy. The techniques which are used to study a particular combinatorial game depend on the options available to the players. The set of all options that can be encountered in any line of play of a game is called the {\em followers} of the game.  If both players have the same options available from any game position, the game is called \emph{impartial}.  The game of \textsc{Nim}, played on heaps of counters, is the prototypical example of an impartial game  \cite{ANW, Siegel}. In 1902, Bouton~\cite{Bouton_02} proved the complete theory of \textsc{Nim}; Sprague \cite{sprague_36,sprague_37} and Grundy \cite{grundy_39} independently proved that every impartial game is equivalent to a nim-heap. Hence, the techniques used to solve \textsc{Nim} can be used to solve all impartial games. In particular, the Sprague-Grundy theory can be used to determine the game value for impartial games, from which we can determine who can win. The value of an impartial game is calculated recursively via the minimum excluded value. The \emph{minimum excluded value}, or \emph{mex}, of a set $S \subsetneq \mathbb{N}$ is the least non-negative integer not included in $S$. The \textit{Grundy value}, or \emph{nim-value}, of an impartial game $H$ is given by $\mathcal{G}(H)=\mex\{ \mathcal{G}(K)\mid K\mbox{ is an option of } H\}$.  For example, $\mex\{\} = 0$, $\mex\{0,1,3,4,5\}=2$, and $\mex\{2,7,9\}=0$. The game tree of an impartial game $G$ is defined recursively.  It has a root at depth 0, labeled by $G$, which is drawn at the top of any diagram showing the game tree. The children of any node labeled by the game $H$ are the rooted game trees of all the options of $H$.  \cref{simple_tree} shows a simple game tree with each node labeled by its nim-value. Leaves have no options, and so the player whose turn it is will lose if they are presented with those positions. The first player can win the game by playing to the node at level 1 labeled by 0.  If the first player mistakenly moves to the node labeled $1$, they will lose when the second player moves to the node at level 2 labeled by $0$.  

\begin{figure}[ht!]
  \begin{center}
  \begin{tikzpicture}
    [mystyle/.style={circle,draw=black,fill=white,line width=1.00,node font = \tiny \bf , minimum size = 5mm},
    mystyle_scale/.style={circle,draw=black,fill=white,line width=2.00,node font = \tiny \bf , minimum size = 5mm}]
    \tikzmath{\y0 = 0; \y1=-1; \y2=-2;}
    \node (a) at (0,\y0) [mystyle] {\small 2};
    \node (b) at (-0.5,\y1) [mystyle] {\small 0};
    \node (c) at (0.5,\y1) [mystyle] {\small 1};
    \node (j) at (0.5,\y2) [mystyle] {\small 0};

    \path[line width=1.50, black] (a) edge (b) edge (c)
    (c) edge (j);
  \end{tikzpicture}
  \end{center}
  \caption{Game tree of an impartial game with nim-value 2. \label{simple_tree}}
  \end{figure}
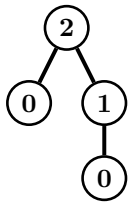

One of the important results in impartial combinatorial game theory is that the winner under optimal play can be determined from the nim-value. 

\begin{theorem}[\cite{ANW}]\label{thm:impartial}
Let $G$ be an impartial game under the normal play winning convention. The game $G$ is won by the second player if and only if the nim-value is zero.
\end{theorem}
\begin{proof}
  The proof is an induction on the height of the game tree.  If a game has no options, it is lost by the first player.  The nim-value of $G$ is the  minimum excluded value of the nim-values of the options of $G$.  If the options are an empty set then $\mathcal{G}(G) = \mex\{\} = 0$.  We assume that the theorem holds for all game trees of heights no more than $h$.  Consider a game tree of height $h+1$.  The heights of the subtrees at each neighbour of the root are no more than $h$ so the winner on each node is determined by its nim-value respectively. If there is a neighbour with nim-value 0, then the induction shows that the first player can win by moving to that position.  Correspondingly, the nim-value at the root must have a positive nim-value because the minimum excluded value is taken over a set containing $0$.  If none of the neighbours of the root have nim-value 0, then induction gives that no matter which play the first player makes, they will lose.  Since no neighbour has nim-value 0, the minimum excluded value over all neighbours will be 0 and this gives that the nim-value of the game at the root is 0.
  \end{proof}
  If a game is won by the second player (also called the \emph{Previous player}), then the game is called a $\mathcal{P}$-position. Otherwise, there exists a winning strategy for the first player (also called a \emph{Next player}) and the game is an $\mathcal{N}$-position. Next we give a detailed example involving these concepts. For more background in combinatorial game theory, see \cite{ANW, Siegel}. 

  \begin{example}\label{p3_pt_ex}
  Let us analyze playing \nodekayles on the disjoint union of a path of length 2 and a disjoint vertex, shown in \cref{p3_pt}.
    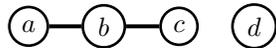
\begin{figure}[ht!]
  \begin{center}
  \begin{tikzpicture}
    [mystyle/.style={circle,draw=black,fill=white,line width=1.00,node font = \tiny \bf , minimum size = 4mm},
    mystyle_scale/.style={circle,draw=black,fill=white,line width=2.00,node font = \tiny \bf , minimum size = 4mm}]
    \node (a) at (0,0) [mystyle] {\small $a$};
    \node (b) at (1,0) [mystyle] {\small $b$};
    \node (c) at (2,0) [mystyle] {\small $c$};
    \node (d) at (3,0) [mystyle] {\small $d$};

    \draw[line width=1.50, black] (a) edge (b);
    \draw[line width=1.50, black] (b) edge (c);

  \end{tikzpicture}
  \end{center}
  \caption{The graph $P_3 \,\dot\cup\, K_1$. \label{p3_pt}}
\end{figure}
Since we need to know the nim-values of all followers, we start by determining the nim-values of some small graphs. \nodekayles played on an empty graph has no moves so nim-value 0.  A single vertex has only one option, the empty graph, and thus has nim-value $\mex\{0\} = 1$. The disjoint union of two vertices again has only one option up to isomorphism, the single vertex which has nim-value 1, thus the nim-value of two disjoint nodes is $\mex\{1\} = 0$.  When playing \nodekayles on just the path $P_3$ (without the additional disjoint vertex), playing on vertex $b$ leaves an empty graph; playing on either $a$ or $c$ leaves a single vertex.  These two options have nim-values 0 and 1 respectively and thus the nim-value of $P_3$ is $\mex\{0,1\} = 2$.
We are now ready to consider playing \nodekayles on $P_3 \,\dot\cup\, K_1$. Playing vertex $a$ or $c$ leaves the disjoint union of two vertices, which has  nim-value 0. Playing vertex $b$ leaves a single vertex, which has  
 nim-value $1$.  Playing vertex $d$ leaves $P_3$, which has nim-value $2$. Thus the nim-value of $P_3 \,\dot\cup\, K_1$ is $\mex\{0,1,2\} = 3$.  The game tree of this game is shown in \cref{p3_pt_tree} with the edges labeled by the vertices played and the nodes labeled with nim-values.
 
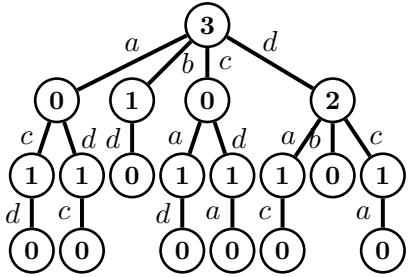
\begin{figure}[ht!]
  \begin{center}
  \begin{tikzpicture}
    [mystyle/.style={circle,draw=black,fill=white,line width=1.00,node font = \tiny \bf , minimum size = 5mm},
    mystyle_scale/.style={circle,draw=black,fill=white,line width=2.00,node font = \tiny \bf , minimum size = 5mm}]
    \tikzmath{\y0 = 0; \y1=-1; \y2=-2; \y3=-3;}
    \node (0) at (0,\y0) [mystyle] {\small 3};
    
    \node (a) at (-2,\y1) [mystyle] {\small 0};
    \node (b) at (-1,\y1) [mystyle] {\small 1};
    \node (c) at (0,\y1) [mystyle] {\small 0};
    \node (d) at (1 + 2/3,\y1) [mystyle] {\small 2};
    
    \node (ac) at (-2-1/3,\y2) [mystyle] {\small 1};
    \node (ad) at (-1-2/3,\y2) [mystyle] {\small 1};
    \node (bd) at (-1,\y2) [mystyle] {\small 0};
    \node (ca) at (-1/3,\y2) [mystyle] {\small 1};
    \node (cd) at (1/3,\y2) [mystyle] {\small 1};
    \node (da) at (1,\y2) [mystyle] {\small 1};
    \node (db) at (1+2/3,\y2) [mystyle] {\small 0};
    \node (dc) at (2+1/3,\y2) [mystyle] {\small 1};
    
    \node (acd) at (-2-1/3,\y3) [mystyle] {\small 0};
    \node (adc) at (-1-2/3,\y3) [mystyle] {\small 0};
    \node (cad) at (-1/3,\y3) [mystyle] {\small 0};
    \node (cda) at (1/3,\y3) [mystyle] {\small 0};
    \node (dac) at (1,\y3) [mystyle] {\small 0};
    \node (dca) at (2+1/3,\y3) [mystyle] {\small 0};

    \draw[line width=1.50, black] (0) -- node[above] {$a$}  (a);
    \draw[line width=1.50, black] (0) -- node[right] {$b$}   (b);
    \draw[line width=1.50, black] (0) -- node[right] {$c$} (c);
    \draw[line width=1.50, black] (0) --node[above] {$d$} (d);
    \draw[line width=1.50, black] (a) --node[left] {$c$}  (ac);
    \draw[line width=1.50, black] (a) --node[right] {$d$} (ad);
    \draw[line width=1.50, black] (ac) --node[left] {$d$} (acd);
    \draw[line width=1.50, black] (ad) --node[left] {$c$}  (adc);
    \draw[line width=1.50, black] (b) --node[left] {$d$}  (bd);
    \draw[line width=1.50, black] (c) --node[left] {$a$}  (ca);
    \draw[line width=1.50, black] (c) --node[right] {$d$} (cd);
    \draw[line width=1.50, black] (ca) --node[left] {$d$} (cad);
    \draw[line width=1.50, black] (cd) --node[left] {$a$}  (cda);
    \draw[line width=1.50, black] (d) --node[left] {$a$}  (da);
    \draw[line width=1.50, black] (d) --node[left] {$b$} (db);
    \draw[line width=1.50, black] (d) --node[right] {$c$} (dc);
    \draw[line width=1.50, black] (da) --node[left] {$c$}  (dac);
    \draw[line width=1.50, black] (dc) --node[left] {$a$} (dca);

  \end{tikzpicture}
  \end{center}
  \caption{Game tree of \nodekayles played on $P_3 \,\dot\cup\, K_1$. \label{p3_pt_tree}}
  \end{figure}
    \end{example}

The paper is presented as follows. In \cref{sec_nim values} we explore playing \nofil on small Steiner triple systems. We compute the nim-values for all $\STS$s of order no more than 15 and report the nim-values for some $\STS(19)$s, $\STS(21)$s, and $\STS(25)$s.  As a special case, we give the entire game tree for the cyclic $\STS(13)$. We also explore nim-values encountered in subpositions as these games are played. Furthermore, we prove a general result for nim-values of designs which are vertex-transitive, \cref{vertextransitive}. 
With the end goal of proving the complexity of \nofil on Steiner triple systems, the rest of the paper examines necessary and sufficient conditions for graph embeddings in Steiner triple systems. In \cref{u_bounds} we derive bounds on the size of the set of unplayable points in terms of the other parameters of the design, vertex partition, and graph. Using this result, in \cref{v_bounds} we derive bounds on $v$. Together, these results give necessary conditions for the graph embedding. Next, in \cref{sufficient}, we derive sufficient conditions for graph embeddings. We prove that any graph can be embedded in a Steiner triple system providing that the order of the design is large enough, \cref{thm_asymp}. We conclude the paper by proving that the complexity of determining the outcome class of \nofil is PSPACE-complete, \cref{complexity}.

\section{Nim-Values of \nofil}\label{sec_nim values}

In this section we investigate the game \nofil played on Steiner triple systems of small orders.  We determine the winner under optimal play and determine the nim-values using SageMath~\cite{sagemath}.  The unique $\STS(7)$ is the projective plane over $\F_2$. The first three points played cannot all be on a block, therefore they must be non-collinear.  The projective linear group acts 3-transitively on non-collinear triples of points \cite{casse_projective_2006}, thus all lines of play are isomorphic up to the third turn of play. The three points on the unique line disjoint from the first three played points are now unplayable and only one point remains available and thus the game ends after the fourth turn and the second player has won.  All second player win games have nim-value 0 by \cref{thm:impartial}.

The unique $\STS(9)$ is the affine plane over $\F_3$. The affine linear group again acts 3-transitively on non-collinear points  \cite{MR1409812}, so just as in the $\STS(7)$ all possible lines of play are isomorphic for the first three moves.  From the game played in \cref{ex: example}, we see that after the third turn the game has become isomorphic to playing \nodekayles on $K_3$ and again all moves are the same up to isomorphism.  Thus the game ends after the fourth turn, the second player has won, and the nim-value is 0.  In playing \nofil on either the $\STS(7)$ or the $\STS(9)$, up to isomorphism, there are no choices for the players to make; every game ends in exactly four moves with the second player winning.

The automorphism groups of these two $\STS$s mean that different lines of play are isomorphic to each other and this simplified the computation of their nim-values.  We can use the transitivity of a design's automorphism group to reduce the possible nim-values of \nofilpunct. If a design has a vertex transitive automorphism group, then all the positions encountered by the second player at the start of the second turn are isomorphic to each other.  If that position, and hence all positions at the start of the second turn, has nim-value $0$, then the nim-value of the initial game is $\mex\{0\} = 1$.  If the nim-value of that position is $x > 0$, then the nim-value of the initial game is $\mex\{x\} = 0$ for any $x >0$.  Thus we have proved the following.
\begin{proposition}\label{vertextransitive}
  If a hypergraph has a vertex-transitive automorphism group, then the nim-value of \nofil played on this hypergraph is either $0$ or $1$.
\end{proposition}
We state this proposition for general hypergraphs so it can be applied after the first turn.  For example, in both the $\STS(7)$ and $\STS(9)$ the automorphism groups are at least 3-transitive on non-collinear points, so the available hypergraphs at the start of the first three moves all have vertex transitive automorphism groups and the proposition applies. The broad utility of this proposition for designs and even Steiner triple systems is limited however because it is known that almost all Steiner triple systems have trivial automorphism groups \cite{MR584402}, so \cref{vertextransitive} is a highly special case.

Using SageMath \cite{sagemath} we have computed that \nofil played on either of the non-isomorphic $\STS(13)$s has nim-value 1.  This means that each is won by the first player if the first player plays optimally.  However, not all lines of play are isomorphic in these $\STS$s.  There is one cyclic $\STS(13)$ with a vertex transitive automorphism group.  It is generated from base blocks $\{0, 3, 4\}$ and $\{0, 5, 7\}$ in $\Z_{13}$.  Using the fact that all first moves are equivalent up to isomorphism allows us to reduce the game-tree enough to display it.  Its structure is shown in \cref{cyclic_sts_13}.  Inside each node is the nim-value of the position rooted at that node. It can be seen that although the first player can win playing optimally, there are opportunities for the first player to lose.  For examples, note that it is the first players turn at the nodes  of even depth.  If the first player moves to any node with non-zero nim-value, they will have offered the second player an opportunity to win.  On the first player's second turn they must move to subtree $C$; the second player can win from any of the other subtrees at depth 3. The other $\STS(13)$ has 4 orbits of vertices and its game tree is too large to display.  The highest nim-values of any of the followers of either $\STS(13)$ is 3.  The positions labeled $B$, $E$, $F$, and $G$ in \cref{cyclic_sts_13} are the positions with nim-value 3.

\begin{figure}[tb!]

  \begin{tikzpicture}
    [mystyle/.style={circle,draw=black,fill=white,line width=1.00,node font = \tiny \bf , minimum size = 5mm},
    mystyle_scale/.style={circle,draw=black,fill=white,line width=2.00,node font = \tiny \bf , minimum size = 5mm}]
    \tikzmath{\y0 = 0; \y1=-1; \y2=-1.8;\y3=-3;\ya2=-2;\ya0=0.5;}
    \node (a) at (0,\y0) [mystyle] {\small 1};
    \node (b) at (0,\y1) [mystyle] {\small 0};
    \node (c) at (-3,\y2) [mystyle] {\small 2};
    \node (j) at (3,\y2) [mystyle] {\small 2};
    \node (d) at (-5.5,\y3) [mystyle] {\small $A$};
    \node (e) at (-4.5,\y3) [mystyle] {\small $B$};
    \node (f) at (-3.5,\y3) [mystyle] {\small $C$};
    \node (g) at (-2.5,\y3) [mystyle] {\small $D$};
    \node (h) at (-1.5,\y3) [mystyle] {\small $E$};
    \node (i) at (-0.5,\y3) [mystyle] {\small $F$};
    \node (ba) at (0.5,\y3) [mystyle] {\small $A$};
    \node (bb) at (1.5,\y3) [mystyle] {\small $B$};
    \node (bc) at (2.5,\y3) [mystyle] {\small $C$};
    \node (bd) at (3.5,\y3) [mystyle] {\small $D$};
    \node (be) at (4.5,\y3) [mystyle] {\small $E$};
    \node (bf) at (5.5,\y3) [mystyle] {\small $G$};

    \path[line width=1.50, black] (c) edge (d) edge (e) edge (f) edge (g) edge (h)
    edge (i)
    (j) edge (ba) edge (bb) edge (bc) edge (bd) edge (be) edge (bf)
    (b) edge (c) edge (j)
    (a) edge (b);

    \node (bounding_box) at (0,-12) [] {};

    \begin{scope} [xshift=-4cm,yshift=-7cm]
      \begin{scope} [transform canvas={scale =0.6}]
        \node () at (-2,\ya0) [font = \bf] {\large \textit{A}};
        \node (Aa) at (-1,\ya0) [mystyle_scale] {\large 1};
        
        \node (Ab) at (-5,\y1) [mystyle_scale] {\large 0};
        \node (Ae) at (-3.5,\y1) [mystyle_scale] {\large 2};
        \node (Ai) at (-1.5,\y1) [mystyle_scale] { \large 0};
        \node (Abd) at (2/3,\y1) [mystyle_scale] { \large 2};
        \node (Abj) at (3,\y1) [mystyle_scale] {\large 2};
        
        \node (Ac) at (-5,\ya2) [mystyle_scale] {\large 1};
        \node (Af) at (-4,\ya2) [mystyle_scale] {\large 1};
        \node (Ah) at (-3,\ya2) [mystyle_scale] {\large 0};
        \node (Aj) at (-2,\ya2) [mystyle_scale] {\large 1};
        \node (Abb) at (-1,\ya2) [mystyle_scale] {\large 1};
        \node (Abe) at (0,\ya2) [mystyle_scale] {\large 1};
        \node (Abg) at (2/3,\ya2) [mystyle_scale] {\large 1};
        \node (Abi) at (4/3,\ya2) [mystyle_scale] {\large 0};
        \node (Aca) at (7/3,\ya2) [mystyle_scale] {\large 1};
        \node (Acc) at (3,\ya2) [mystyle_scale] {\large 1};
        \node (Ace) at (11/3,\ya2) [mystyle_scale] {\large 0};
        
        \node (Ad) at (-5,\y3) [mystyle_scale] {\large 0};
        \node (Ag) at (-4,\y3) [mystyle_scale] {\large 0};
        \node (Aba) at (-2,\y3) [mystyle_scale] {\large 0};
        \node (Abc) at (-1,\y3) [mystyle_scale] {\large 0};
        \node (Abf) at (0,\y3) [mystyle_scale] {\large 0};
        \node (Abh) at (2/3,\y3) [mystyle_scale] {\large 0};
        \node (Acb) at (7/3,\y3) [mystyle_scale] {\large 0};
        \node (Acd) at (3,\y3) [mystyle_scale] {\large 0};

        \path[line width=1.50, black] (Ac) edge (Ad)
	(Ab) edge (Ac)
	(Af) edge (Ag)
	(Ae) edge (Af) edge (Ah)
	(Aj) edge (Aba)
	(Abb) edge (Abc)
	(Ai) edge (Aj) edge (Abb)
	(Abe) edge (Abf)
	(Abg) edge (Abh)
	(Abd) edge (Abe) edge (Abg) edge (Abi)
	(Aca) edge (Acb)
	(Acc) edge (Acd)
	(Abj) edge (Aca) edge (Acc) edge (Ace)
	(Aa) edge (Ab) edge (Ae) edge (Ai) edge (Abd) edge (Abj);
      \end{scope}
    \end{scope}

    \begin{scope} [xshift=5cm,yshift=-7cm]
      \begin{scope} [transform canvas={scale =0.6}]
        \node () at (-1,\ya0) [font = \bf] {\large \textit{B}};
        \node at (0,\ya0) [mystyle_scale] (Ba) {\large 3};
        
        \node at (-3,\y1) [mystyle_scale] (Bb) {\large 0};
        \node at (-1.5,\y1) [mystyle_scale] (Bg) {\large 1};
        \node at (0,\y1) [mystyle_scale] (Bi) {\large 2};
        \node at (1.5,\y1) [mystyle_scale] (Bbc) {\large 0};
        \node at (2.5+2/3,\y1) [mystyle_scale] (Bbf) {\large 2};
        
        \node at (-3.5,\ya2) [mystyle_scale] (Bc) {\large 1};
        \node at (-2.5,\ya2) [mystyle_scale] (Be) {\large 1};
        \node at (-1.5,\ya2) [mystyle_scale] (Bh) {\large 0};
        \node at (-0.5,\ya2) [mystyle_scale] (Bj) {\large 1}; 
        \node at (0.5,\ya2) [mystyle_scale] (Bbb) {\large 0};
        \node at (1.5,\ya2) [mystyle_scale] (Bbd) {\large 1};
        \node at (2.5,\ya2) [mystyle_scale] (Bbg) {\large 1};
        \node at (2.5+2/3,\ya2) [mystyle_scale] (Bbi) {\large 1};
        \node at (2.5+4/3,\ya2) [mystyle_scale] (Bca) {\large 0};
        
        \node at (-3.5,\y3) [mystyle_scale] (Bd) {\large 0};
        \node at (-2.5,\y3) [mystyle_scale] (Bf) {\large 0};
        \node at (-0.5,\y3) [mystyle_scale] (Bba) {\large 0};
        \node at (1.5,\y3) [mystyle_scale] (Bbe) {\large 0};
        \node at (2.5,\y3) [mystyle_scale] (Bbh) {\large 0};
        \node at (2.5+2/3,\y3) [mystyle_scale] (Bbj) {\large 0};

        \path[line width =1.5, black] (Bc) edge (Bd)
	(Be) edge (Bf)
	(Bb) edge (Bc) edge (Be)
	(Bg) edge (Bh)
	(Bj) edge (Bba)
	(Bi) edge (Bj) edge (Bbb)
	(Bbd) edge (Bbe)
	(Bbc) edge (Bbd)
	(Bbg) edge (Bbh)
	(Bbi) edge (Bbj)
	(Bbf) edge (Bbg) edge (Bbi) edge (Bca)
	(Ba) edge (Bb) edge (Bg) edge (Bi) edge (Bbc) edge (Bbf);
      \end{scope}
    \end{scope}

    \begin{scope} [xshift=-5cm,yshift=-12cm]
      \begin{scope} [transform canvas={scale =0.6}]
        \node () at (-1,\ya0) [font = \bf] {\large \textit{C}};

        \node at (0,\ya0) [mystyle_scale] (a) {\large 0};

        \node at (-3-1/3,\y1) [mystyle_scale] (b) {\large 1};
        \node at (-1.5-1/3,\y1) [mystyle_scale] (d) {\large 2};
        \node at (1/3,\y1) [mystyle_scale] (h) {\large 2};
        \node at (2+2/3,\y1) [mystyle_scale] (bd) {\large 2};

        \node at (-3-1/3,\ya2) [mystyle_scale] (c) {\large 0};
        \node at (-2-1/3,\ya2) [mystyle_scale] (e) {\large 1};
        \node at (-1-1/3,\ya2) [mystyle_scale] (g) {\large 0};
        \node at (-1/3,\ya2) [mystyle_scale] (i) {\large 1};
        \node at (1/3,\ya2) [mystyle_scale] (ba) {\large 1};
        \node at (1,\ya2) [mystyle_scale] (bc) {\large 0};
        \node at (2,\ya2) [mystyle_scale] (be) {\large 1};
        \node at (2+2/3,\ya2) [mystyle_scale] (bg) {\large 1};
        \node at (3+1/3,\ya2) [mystyle_scale] (bi) {\large 0};

        \node at (-2-1/3,\y3) [mystyle_scale] (f) {\large 0};
        \node at (-1/3,\y3) [mystyle_scale] (j) {\large 0};
        \node at (1/3,\y3) [mystyle_scale] (bb) {\large 0};
        \node at (2,\y3) [mystyle_scale] (bf) {\large 0};
        \node at (2+2/3,\y3) [mystyle_scale] (bh) {\large 0};

        \path[line width=1.5, black] (b) edge (c)
	(e) edge (f)
	(d) edge (e) edge (g)
	(i) edge (j)
	(ba) edge (bb)
	(h) edge (i) edge (ba) edge (bc)
	(be) edge (bf)
	(bg) edge (bh)
	(bd) edge (be) edge (bg) edge (bi)
	(a) edge (b) edge (d) edge (h) edge (bd);

      \end{scope}
    \end{scope}

    \begin{scope} [xshift=+5cm,yshift=-12cm]
      \begin{scope} [transform canvas={scale =0.6}]
        \node () at (-1,\ya0) [font = \bf] {\large \textit{D}};
        
        \node at (0,\ya0) [mystyle_scale] (a) {\large 1};

        \node at (-2.5,\y1) [mystyle_scale] (b) {\large 0};
        \node at (-0.5,\y1) [mystyle_scale] (g) {\large 2};
        \node at (1,\y1) [mystyle_scale] (ba) {\large 0};
        \node at (2+2/3,\y1) [mystyle_scale] (bd) {\large 2};
        
        \node at (-3,\ya2) [mystyle_scale] (c) {\large 1};
        \node at (-2,\ya2) [mystyle_scale] (e) {\large 1};
        \node at (-1,\ya2) [mystyle_scale] (h) {\large 1};
        \node at (0,\ya2) [mystyle_scale] (j) {\large 0};
        \node at (1,\ya2) [mystyle_scale] (bb) {\large 1};
        \node at (2,\ya2) [mystyle_scale] (be) {\large 1};
        \node at (2+2/3,\ya2) [mystyle_scale] (bg) {\large 1};
        \node at (3+1/3,\ya2) [mystyle_scale] (bi) {\large 0};
        
        \node at (-3,\y3) [mystyle_scale] (d) {\large 0};
        \node at (-2,\y3) [mystyle_scale] (f) {\large 0};
        \node at (-1,\y3) [mystyle_scale] (i) {\large 0};
        \node at (1,\y3) [mystyle_scale] (bc) {\large 0};
        \node at (2,\y3) [mystyle_scale] (bf) {\large 0};
        \node at (2+2/3,\y3) [mystyle_scale] (bh) {\large 0};

        \path[line width=1.5, black] (c) edge (d)
	(e) edge (f)
	(b) edge (c) edge (e)
	(h) edge (i)
	(g) edge (h) edge (j)
	(bb) edge (bc)
	(ba) edge (bb)
	(be) edge (bf)
	(bg) edge (bh)
	(bd) edge (be) edge (bg) edge (bi)
	(a) edge (b) edge (g) edge (ba) edge (bd);
      \end{scope}
    \end{scope}

    \begin{scope} [xshift=-5cm,yshift=-17cm]
      \begin{scope} [transform canvas={scale =0.6}]

        \node () at (-1,\ya0) [font = \bf] {\large \textit{E}};
        \node at (0,\ya0) [mystyle_scale] (a) {\large 3};
        
        \node at (-2.5,\y1) [mystyle_scale] (b) {\large 1};
        \node at (-1,\y1) [mystyle_scale] (d) {\large 2};
        \node at (0.5,\y1) [mystyle_scale] (h) {\large 0};
        \node at (1.5+2/3,\y1) [mystyle_scale] (ba) {\large 2};

        \node at (-2.5,\ya2) [mystyle_scale] (c) {\large 0};
        \node at (-1.5,\ya2) [mystyle_scale] (e) {\large 1};
        \node at (-0.5,\ya2) [mystyle_scale] (g) {\large 0};
        \node at (0.5,\ya2) [mystyle_scale] (i) {\large 1};
        \node at (1.5,\ya2) [mystyle_scale] (bb) {\large 1};
        \node at (1.5+2/3,\ya2) [mystyle_scale] (bd) {\large 1};
        \node at (1.5+4/3,\ya2) [mystyle_scale] (bf) {\large 0};

        \node at (-1.5,\y3) [mystyle_scale] (f) {\large 0};
        \node at (0.5,\y3) [mystyle_scale] (j) {\large 0};
        \node at (1.5,\y3) [mystyle_scale] (bc) {\large 0};
        \node at (1.5+2/3,\y3) [mystyle_scale] (be) {\large 0};

        \path[line width = 1.5, black] (b) edge (c)
	(e) edge (f)
	(d) edge (e) edge (g)
	(i) edge (j)
	(h) edge (i)
	(bb) edge (bc)
	(bd) edge (be)
	(ba) edge (bb) edge (bd) edge (bf)
	(a) edge (b) edge (d) edge (h) edge (ba);
        
      \end{scope}
    \end{scope}

    \begin{scope} [xshift=1cm,yshift=-17cm]
      \begin{scope} [transform canvas={scale =0.6}]
        
        \node () at (-1,\ya0) [font = \bf] {\large \textit{F}};
        \node at (0,\ya0) [mystyle_scale] (a) {\large 3};

        \node at (-1.5,\y1) [mystyle_scale] (b) {\large 0};
        \node at (-0.5,\y1) [mystyle_scale] (e) {\large 1};
        \node at (0.5+2/3,\y1) [mystyle_scale] (g) {\large 2};
        
        \node at (-1.5,\ya2) [mystyle_scale] (c) {\large 1};
        \node at (-0.5,\ya2) [mystyle_scale] (f) {\large 0};
        \node at (0.5,\ya2) [mystyle_scale] (h) {\large 1};
        \node at (0.5+2/3,\ya2) [mystyle_scale] (j) {\large 1};
        \node at (0.5+4/3,\ya2) [mystyle_scale] (bb) {\large 0};
        
        \node at (-1.5,\y3) [mystyle_scale] (d) {\large 0};
        \node at (0.5,\y3) [mystyle_scale] (i) {\large 0};
        \node at (0.5+2/3,\y3) [mystyle_scale] (ba) {\large 0};

        \path[line width = 1.5, black] (c) edge (d)
	(b) edge (c)
	(e) edge (f)
	(h) edge (i)
	(j) edge (ba)
	(g) edge (h) edge (j) edge (bb)
	(a) edge (b) edge (e) edge (g);
      \end{scope}
    \end{scope}

    \begin{scope} [xshift=6cm,yshift=-17cm]
      \begin{scope} [transform canvas={scale =0.6}]
       	
        \node () at (-1,\ya0) [font = \bf] {\large \textit{G}};
        \node at (0,\ya0) [mystyle_scale] (a) {\large 3};

        \node at (-1.5,\y1) [mystyle_scale] (b) {\large 1};
        \node at (0,\y1) [mystyle_scale] (d) {\large 2};
        \node at (1.5,\y1) [mystyle_scale] (h) {\large 0};
        
        \node at (-1.5,\ya2) [mystyle_scale] (c) {\large 0};
        \node at (-0.5,\ya2) [mystyle_scale] (e) {\large 1};
        \node at (0.5,\ya2) [mystyle_scale] (g) {\large 0};
        \node at (1.5,\ya2) [mystyle_scale] (i) {\large 1};
        
        \node at (-0.5,\y3) [mystyle_scale] (f) {\large 0};
        \node at (1.5,\y3) [mystyle_scale] (j) {\large 0};

        \path[line width=1.5, black] (b) edge (c)
	(e) edge (f)
	(d) edge (e) edge (g)
	(i) edge (j)
	(h) edge (i)
	(a) edge (b) edge (d) edge (h);
      \end{scope}
    \end{scope}

  \end{tikzpicture}
  \caption{The game-tree of \nofil played on the cyclic $\STS(13)$. \label{cyclic_sts_13}}
\end{figure}
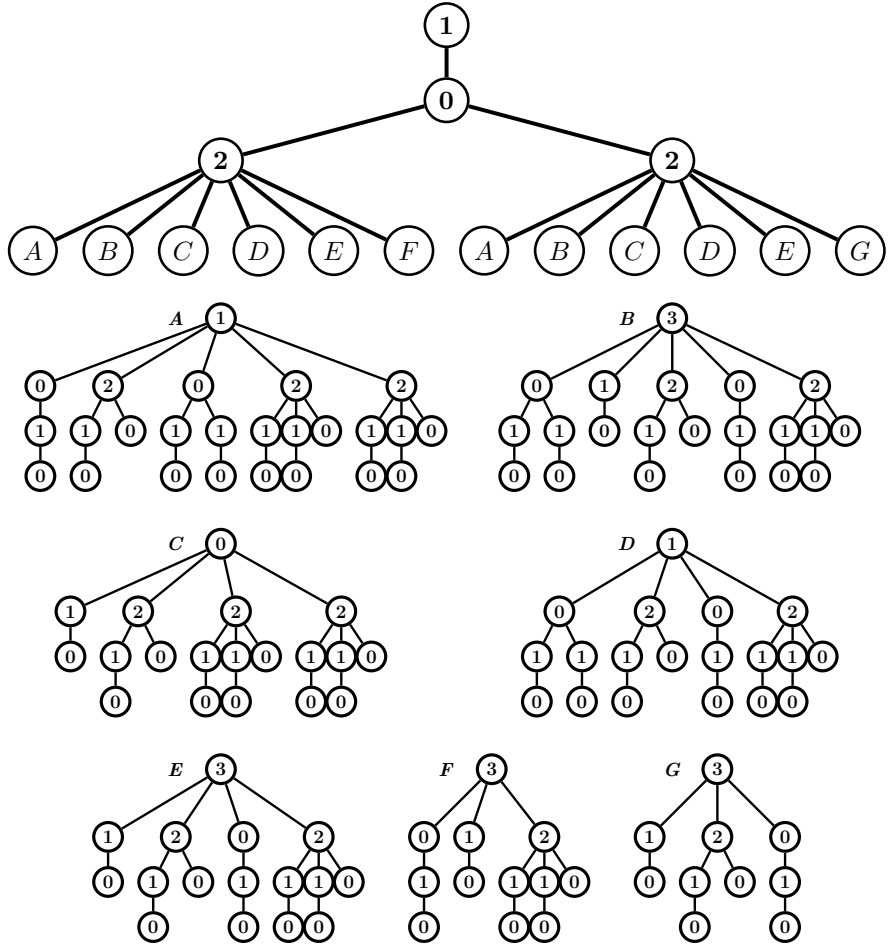

We also computed that of the 80 non-isomorphic $\STS(15)$s, 73 have nim-value 0, one has nim-value 1, five have nim-value 2, and one has nim-value 3.  The four of these that are resolvable all have nim-value 0. Using hill climbing \cite{MR833797} we generated 2000 distinct  $\STS(19)$; all that we generated have nim-value 1.  Similarly we generated 1000 distinct $\STS(21)$; $941$ had nim-value 0 and $59$ had nim-value 2.  We generated 400 distinct $\STS(25)$; 245 had nim-value 1 and 155 had nim-value 3. From these experiments there is an apparent pattern that odd and even nim-values predominate for $v \equiv 1,3 \pmod 6$ respectively.  The average time to compute the nim-value of an $\STS(19)$ with a SageMath program is $102s$, so exhaustively determining the nim-value for all 11084874829 is unfortunately infeasible this way.

\begin{figure}[ht!]
  \begin{center}
  \begin{tikzpicture}
    [mystyle/.style={circle,draw=black,fill=white,line width=1.00,node font = \tiny \bf , minimum size = 6mm},
    mystyle_scale/.style={circle,draw=black,fill=white,line width=2.00,node font = \tiny \bf , minimum size = 6mm}]
    \node (a) at (0:0) [mystyle] {\small 6};
    \node (b) at (0:1) [mystyle] {\small 3};
    \node (c) at (0:2) [mystyle] {\small 4};
    \node (d) at (120:1) [mystyle] {\small 9};
    \node (e) at (240:1) [mystyle] {\small 14};

    \draw[line width=1.50, black] (a) edge (b);
    \draw[line width=1.50, black] (b) edge (c);
    \draw[line width=1.50, black] (a) edge (d);
    \draw[line width=1.50, black] (a) edge (e);

  \end{tikzpicture}
  \end{center}
  \caption{Smallest available hypergraph with nim-value 3 encountered when playing \nofil on $\STS$ up to order 15. \label{small_3}}
  \end{figure}
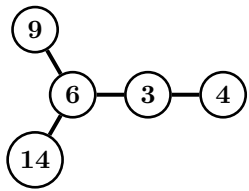

The smallest available hypergraph with nim-value 2 is the 2-path.  This is encountered when playing \nofil on the cyclic $\STS(13)$ and is the middle subtree of position $B$ in \cref{cyclic_sts_13}.  The smallest available hypergraph with nim-value 3 encountered when playing \nofil on Steiner triple systems of orders no more than 15 is shown in \cref{small_3}, with the nodes labeled as in the embedding shown in \cref{grundy_3_sts_15}. Because some points are represented by two digit numbers we list blocks as sets. It does not appear when playing on either of the $\STS(13)$.  It is a subposition in the games played on 39 of the 80 non-isomorphic $\STS(15)$s.  We give one example in \cref{grundy_3_sts_15}. In this example playing vertices $9$ or $14$ gives the disjoint union of an edge and a vertex, which has nim-value 0.  Playing vertex $3$ leaves two disjoint vertices, which has nim-value 0.  Playing vertex $6$ gives a single vertex, which has nim-value 1. Playing vertex $4$ leaves a path, which has nim-value 2.  There is one hypergraph on fewer points with nim-value 3: the disjoint union of a path on three vertices, $P_3$, and $K_1$ discussed in \cref{p3_pt_ex}.  It does not occur as a subposition of \nofil played on any $\STS$ of order 15 or smaller, but it does occur as a subposition played on an $\STS(19)$.  In \cref{sec_restrictions} we will see that some counting arguments show that 19 is the smallest possible order of an $\STS$ which could contain this disjoint union.  The smallest available hypergraph with nim-value 4 found amongst the Steiner triple systems with order no more than 15 is shown in \cref{grundy_4_sts_15}. It only occurs in two of the 80 non-isomorphic $\STS(15)$s. No available hypergraph with nim-value 5 is found when playing \nofil on the $\STS$s with order no more than 15.

\renewcommand{\arraystretch}{1.5}
\begin{table}[ht!]
\begin{center}
  \begin{tabular}{|c|c|}\hline
    $P$ & $\overline{0}$, $\overline{2}$, $\overline{7}$, $\overline{12}$ \\\hline
    $A$ & $\underline{3}$, $\underline{4}$, $\underline{6}$, $\underline{9}$, $\underline{14}$ \\ \hline
    $U$ & 1, 5, 8, 10, 11, 13 \\ \hline
    \multirow{2}{*}{$PPU$} & \{$\overline{0}$, $\overline{2}$, 1\}, \{$\overline{0}$, $\overline{7}$, 8\}, \{$\overline{0}$, $\overline{12}$, 11\} \\
    & \{$\overline{2}$, $\overline{7}$, 13\}, \{$\overline{2}$, $\overline{12}$, 10\}, \{$\overline{7}$, $\overline{12}$, 5\}\\ \hline
  $PAA$ & \{$\overline{0}$, $\underline{3}$, $\underline{4}$\}, \{$\overline{2}$, $\underline{3}$, $\underline{6}$\}, \{$\overline{7}$, $\underline{6}$,  $\underline{14}$\}, \{$\overline{12}$, $\underline{6}$, $\underline{9}$\} \\ \hline
   \multirow{3}{*}{$PAU$ }&
                            \{$\overline{0}$, $\underline{6}$, 5\}, \{$\overline{0}$, $\underline{9}$, 10\}, \{$\overline{0}$,  $\underline{14}$, 13\}, \{$\overline{7}$, $\underline{9}$, 1\}, \\
    & \{$\overline{12}$,  $\underline{14}$, 1\},     \{$\overline{2}$, $\underline{4}$, 5\},    \{$\overline{2}$, $\underline{9}$, 8\},    \{$\overline{2}$,  $\underline{14}$, 11\}, \\
&    \{$\overline{7}$, $\underline{3}$, 10\},    \{$\overline{12}$, $\underline{3}$, 13\},    \{$\overline{7}$, $\underline{4}$, 11\},    \{$\overline{12}$, $\underline{4}$, 8\} \\ \hline
$PUU$ & $\emptyset$ \\ \hline
    \multirow{2}{*}{$AAU$} & \{$\underline{4}$, $\underline{6}$, 1\}, \{$\underline{3}$,  $\underline{14}$, 8\}, \{$\underline{3}$, $\underline{9}$, 11\}, \\
&     \{$\underline{4}$, $\underline{9}$, 13\},  \{$\underline{4}$,  $\underline{14}$, 10\}, \{$\underline{9}$,  $\underline{14}$, 5\} \\ \hline
 $AUU$ & \{$\underline{3}$, 1, 5\}, \{$\underline{6}$, 8, 13\}, \{$\underline{6}$, 10, 11\} \\ \hline
$UUU$ & \{1, 8, 11\}, \{1, 10, 13\}, \{5, 8, 10\}, \{5, 11, 13\} \\ \hline
\end{tabular}
  \caption{$\STS(15)$ with an available hypergraph with nim-value 3 embedded. \label{grundy_3_sts_15}}
\end{center}
\end{table}

\renewcommand{\arraystretch}{1}

\renewcommand{\arraystretch}{1.5}
\begin{table}[ht!]
\begin{center}
  \begin{tabular}{|c|c|}\hline
$P$ &  $\overline{0}$, $\overline{7}$, $\overline{11}$ \\ \hline
$A$ & $\underline{1}$, $\underline{2}$, $\underline{3}$, $\underline{5}$, $\underline{6}$, $\underline{9}$, $\underline{10}$, $\underline{13}$, $\underline{14}$ \\ \hline
   $U$ &  4, 8, 12 \\ \hline
  $PPU$ & \{$\overline{0}$, $\overline{7}$, 8\}, \{$\overline{0}$, $\overline{11}$, 12\}, \{$\overline{7}$, $\overline{11}$, 4\} \\ \hline
   \multirow{3}{*}{$PAA$} &  \{$\overline{0}$, $\underline{1}$, $\underline{2}$\},     \{$\overline{0}$, $\underline{5}$, $\underline{6}$\},     \{$\overline{0}$,  $\underline{9}$, $\underline{10}$\},     \{$\overline{0}$, $\underline{13}$, $\underline{14}$\} \\ 
    & \{$\overline{7}$, $\underline{1}$,  $\underline{9}$\},     \{$\overline{7}$, $\underline{2}$, $\underline{13}$\},     \{$\overline{11}$, $\underline{2}$, $\underline{14}$\},     \{$\overline{7}$, $\underline{3}$, $\underline{10}$\} \\ 
 &     \{$\overline{11}$, $\underline{3}$,  $\underline{9}$\},     \{$\overline{11}$, $\underline{5}$, $\underline{13}$\},     \{$\overline{7}$, $\underline{6}$, $\underline{14}$\},     \{$\overline{11}$, $\underline{6}$, $\underline{10}$\}\\ \hline
   $PAU$ & \{$\overline{0}$, $\underline{3}$, 4\}, \{$\overline{11}$, $\underline{1}$, 8\}, \{$\overline{7}$, $\underline{5}$, 12\} \\ \hline
   $PUU$ & $\emptyset$ \\ \hline
   $AAA$ & \{$\underline{1}$, $\underline{3}$, $\underline{5}$\}, \{$\underline{1}$, $\underline{10}$, $\underline{13}$\}, \{$\underline{2}$, $\underline{3}$, $\underline{6}$\}, \{$\underline{5}$,  $\underline{9}$, $\underline{14}$\} \\ \hline
     \multirow{3}{*}{$AAU$} &     \{$\underline{1}$, $\underline{6}$, 4\},     \{$\underline{1}$, $\underline{14}$, 12\},     \{$\underline{2}$, $\underline{5}$, 4\},     \{$\underline{2}$,  $\underline{9}$, 8\} \\ 
&     \{$\underline{2}$, $\underline{10}$, 12\},     \{$\underline{3}$, $\underline{14}$, 8\},     \{$\underline{3}$, $\underline{13}$, 12\},     \{$\underline{9}$, $\underline{13}$, 4\} \\ \
&     \{$\underline{10}$, $\underline{14}$, 4\},     \{$\underline{5}$, $\underline{10}$, 8\},     \{$\underline{6}$, $\underline{13}$, 8\},     \{$\underline{6}$,  $\underline{9}$, 12\} \\ \hline
   $AUU$ & $\emptyset$ \\ \hline
   $UUU$ & \{4, 8, 12\} \\ \hline
\end{tabular}
  \caption{$\STS(15)$ with an available hypergraph with nim-value 4 embedded. \label{grundy_4_sts_15}}
\end{center}
\end{table}
\renewcommand{\arraystretch}{1}

\section{Necessary conditions for embedding a graph in an $\STS$}\label{sec_restrictions}

Some of the interesting examples of nim-values amongst the subpositions of the games on small Steiner triple systems corresponded to available hypergraphs that were graphs.  In addition, \textsc{Node Kayles} is well understood on graphs and offers an avenue to begin to understand aspects of \nofilpunct.  In this section we explore necessary conditions for a graph to appear as an available hypergraph when playing \nofil on an $\STS$.

For a given graph $G$, we are interested in knowing if it can be embedded in an $\STS$ and what restrictions such an $\STS$ has, such as the possible orders of the triple system.  The vertex set of the $\STS$ will be partitioned into $P$, $A$, and $U$ where $A = V(G)$. Although in the discussion of the available hypergraph where the sets $P$, $A$, and $U$ are defined for every turn of game-play, in the remainder of the paper they will be the specific sets for the embedding of the particular graph under consideration.

We classify each block of the $\STS$ according to which of $P$, $A$, and $U$ its three points belong.  Recall that every pair of points is on a unique block of the Steiner triple system.  The rules of the game forbid a block with three played points, \ie  $B \not\subseteq P$ for any block $B$.  If play has led to the turn where the game is equivalent to \nodekayles on a graph, then no block contains three available points, \ie $B \not\subseteq A$ for any block $B$.  Since points in $U$ are determined precisely by being the third point on a block with two played points, there will be no blocks with two played points and one available point. The remaining possible combinations of point types on a block are either one point from $P$ and two from $A$ or any combination  which contains at least one point of $U$. Let
\begin{align*}
  PPU &= \{ B: |B\cap P|=2, |B\cap U| = 1\}, \\
  PAA &= \{ B: |B\cap P|=1, |B\cap A| = 2\}, 
\end{align*}
and define $PAU$, $PUU$, $AAU$, $AUU$, and $UUU$ analogously. 

Since every pair of points from $P$ must be on a block of $PPU$ we have that the number of this type of block is
\[
  |PPU| = \binom{|P|}{2}.
  \]
Consider an edge $\{a_1, a_2\}$ in $G$ and let $\{a_1,a_2,x\}$ be the unique block containing $a_1$ and $a_2$.  If $x \in U$ then no restrictions on playing $a_1$ nor $a_2$ ever are derived from this block because it can never be completely filled in and thus it is not an edge in the available hypergraph.  Thus $x \in P$ and every edge of $G$ is on a block in $PAA$.  If $a_1$ and $a_2$ are not the endpoints of an edge in $G$ and $x \in P$ then the rules forbid playing both $a_1$ and $a_2$ which contradicts $a_1$ and $a_2$ being adjacent.  Thus no non-edge of $G$ is on a block of type $PAA$ so
\[
   |PAA| = |E(G)|.
 \]
 These account for $2|PAA|$ pairs of $P \times A$ and the rest of these pairs must be on blocks of $PAU$, thus
 \begin{align*}
     |PAU| &= |P||A| -2|PAA| \\
        &=  |P||A| -2|E(G)|.
\end{align*}
Pairs of vertices from $P \times U$ can only be on blocks of $PPU$, $PAU$, or $PUU$, so
\begin{align*}
 |PUU| &= \frac{|P||U|-|PAU|-2|PPU|}{2} \\
        &= \frac{|P|(|U|-|P|-|A|+1)}{2} +|E(G)|.
\end{align*}
Every non-edge in $G$ is on a block of $AAU$ and all other pairs of points from $A$ are on blocks of $PAA$.  Thus
\begin{align*}
 |AAU| &= |E(\overline{G})|\\
        &= \frac{|A|(|A|-1)}{2} - |E(G)|.
\end{align*}
Pairs of points from $A \times U$ are on blocks of $PAU$, $AAU$, or $AUU$. Thus we have
\begin{align*}
|AUU| &= \frac{|U||A|-|PAU|-2|AAU|}{2}\\
        &= \frac{|A|(|U|-|P|-|A|+1)}{2} + 2|E(G)|.
\end{align*}
All the unaccounted pairs of vertices in $U$ must be on blocks of $UUU$.  So we get that
\begin{align*}
|UUU| &= \frac{{|U| \choose 2} - |AUU|-|PUU|}{3} \\
        &= \frac{|U|^2 - |U| -(|A|+|P|)(|U|-|P|-|A|+1)}{6} - |E(G)|.
\end{align*}
To make these expressions more tidy and ease calculations done with them we define $p = |P|$, $a = |A| = |V(G)|$, $u=|U|$, and $e = |E(G)|$. In summary, we have
\begin{align*}
|PPU| &= \binom{p}{2}\\
  |PAA| &= e\\
  |PAU| &=  pa-2e \\
  |PUU| &= \frac{p(u-p-a+1)}{2} +e\\
  |AAU| &= \binom{a}{2} - e \\
  |AUU| &= \frac{a(u-p-a+1)}{2} + 2e \\
  |UUU| &= \frac{u^2 - u -(a+p)(u-p-a+1)}{6} - e.
\end{align*}

Note that the graph with zero vertices can always be embedded into any design and corresponds exactly to the available graph at the end of any line of game play. From now on we assume that the graph $G$ on the available points inherited from the $\STS$ is non-empty, so that $|V(G)|\geq 1$.

We now derive bounds on the size of $U$ in terms of the other parameters.

\begin{lemma}\label{u_bounds}
Let $(X,\mathcal{B})$ be an $\STS(v)$ with points $P$ played.  Suppose the available hypergraph when playing \nofil is the graph $G$.  Then
\begin{align}
u &\leq v-a-\chi'(G),    \label{chi_g} \\
  u &\geq \chi'(\overline{G}), \label{chi_comp_g} \\
  u &\geq \frac{v-a-1}{2}, \label{chi_K_p} \\
u &\geq \frac{a(v-1) -4e}{2a},  \label{auu} \\
u &\geq \frac{3v-2a-1 - \sqrt{(v-2a+1)^2 + 16e}}{4}, \label{puu_1} \\
u &\leq \frac{3v-2a-1 + \sqrt{(v-2a+1)^2 + 16e}}{4},    \label{puu_2} \\
  u &\leq v-a - \frac{2e}{a},  \label{pau} \\
  u &\leq v-a + \frac{1-\sqrt{8(v-a)+1}}{2},  \label{u_binom_p_2} 
\end{align}
and if $v^2-4v \leq 24e$, then either
\begin{equation}\label{uuu_1}
u \geq \frac{v}{2} + \frac{\sqrt{72e-3v^2+12v}}{6}
\end{equation}
or
\begin{equation}\label{uuu_2}
u \leq \frac{v}{2} - \frac{\sqrt{72e-3v^2+12v}}{6}.
\end{equation}

\end{lemma}
\begin{proof}
 Let $PAA = \{\{x,y,\phi(x,y)\}:\{x,y\} \in E(G)\}$. Then the fact that these triples must be edge disjoint implies that $\phi:E(G) \rightarrow P$ is a proper edge colouring of $G$.  Thus $\chi'(G) \leq p$ and \cref{chi_g} follows when we substitute $p = v-a-u$.  Similarly $AAU= \{\{x,y,f(x,y)\}: \{x,y\} \not\in E(G)\}$ gives that $f$ is a proper edge colouring of $\overline{G}$ and results in \cref{chi_comp_g}. Using $PPU = \{\{x,y,g(x,y)\}:x,y \in P\}$ yields an edge colouring of $K_p$. Then $p=v-u-a$ and the fact that any proper edge colouring of $K_p$ has at least $p-1$ colours gives \cref{chi_K_p}. 
  
  Since each of $|PAU|$, $|AUU|$, $|PUU|$, and $|UUU|$ are cardinalities of block sets, they must each be non-negative.  \cref{auu} derives from $|AUU| \geq 0$ and \cref{pau} derives from $|PAU| \geq 0$.

  For \cref{puu_1,puu_2}, we start with $|PUU|\geq 0$, which gives
  \[
    2u^2 + (2a-3v+1)u + (v-a)(v-1) -2e \leq 0.
    \]
Thus $u$ must be between the two roots, which gives the two inequalities.
  
  For \cref{u_binom_p_2} we first note that since every point in $U$ must be on at least one block from $PPU$ and $|PPU| = \binom{p}{2}$, we have $u \leq \binom{p}{2}$.  Using $p = v-a-u$ this implies
  \[
    u^2 + (-2v+2a-1)u + (v-a)(v-a-1) \geq 0.
  \]
The discriminant $8(v-a)+1$ is positive because $v \geq a$.  This means that either   
\begin{align*}
    u &\geq v-a+ \frac{1 + \sqrt{8(v-a)+1}}{2}\text{ or }\\
    u &\leq v-a + \frac{1 - \sqrt{8(v-a)+1}}{2}
  \end{align*}
must hold. The first of these implies that $p + (1+\sqrt{(8(p+u)+1})/2 \leq 0$, which is impossible, so the second must hold.

  Finally, $|UUU| \geq 0$ yields
  \[
    3u^2-3vu + v^2-v-6e \geq 0.
    \]
    If $v^2-4v-24e > 0$ (equivalently $v \geq 2 + \sqrt{24e+4}$), then the discriminant is negative and the inequality is always satisfied.  However, when $v^2-4v \leq 24e$, then $u$ must not lie between the two roots and we have that either
 \begin{align*}
u &\geq \frac{v}{2} + \frac{\sqrt{72e-3v^2+12v}}{6}\text{ or } \\
u &\leq \frac{v}{2} - \frac{\sqrt{72e-3v^2+12v}}{6}
 \end{align*}
must hold.
  \end{proof}

These bounds on $u$ can gives us bounds on $v$.
\begin{theorem} \label{v_bounds}
  The lower bounds in \cref{chi_K_p,auu,puu_1} and $u \geq 0$ are not greater than the upper bounds in \cref{puu_2,pau,u_binom_p_2} exactly when
  \begin{equation}\label{eqn_v_bounds}
    v \geq \begin{cases}
      a -1 + \frac{4e}{a} & a \geq 4 \mbox{ and } e \geq \frac{a^2}{4}, \\
      a -1 + \frac{4e}{a} & a \leq 4 \mbox{ and } e \geq a, \\
      2a-1 & a \geq 4 \mbox { and } \frac{-a+\sqrt{8a^3-7a^2}}{4} \leq e \leq \frac{a^2}{4}, \\
      2a + 2 - \frac{4e}{a} + \sqrt{8a+1 -\frac{32e}{a}} & a \geq 4 \mbox{ and } e \leq  \frac{-a+\sqrt{8a^3-7a^2}}{4}, \\
      a+3 & a \leq 4 \mbox{ and } \frac{a^2}{4} \leq e \leq a, \\
      
      2a + 2 - \frac{4e}{a} + \sqrt{8a+1 -\frac{32e}{a}} & a \leq 4 \mbox{ and } e \leq  \frac{a^2}{4}, \\
    \end{cases}
    \end{equation}
  with 60 exceptional triples $(a,e,v)$ which cannot occur, given in \cref{5_8_not_4_8_all_a}.
\end{theorem}  
\begin{proof}

  We consider all the pairs of lower bounds and upper bounds on $u$ from \cref{u_bounds} and additionally $u \geq 0$.  For each pair we determine the restrictions on $v$ implied by assuming the lower bound is no larger than the upper bound. We do not consider either the fact that $v$ must be an integer, nor that $v \equiv 1,3 \pmod 6$ (which is necessary when $v$ is the order of an $\STS$). The order in which we consider the pairs of bounds is chosen to give information early that is useful when considering later pairs.  We will find that, except for a finite number of exceptional values of $(a,e,v)$, there are only four restrictions on $v$ that are not always exceeded by other restrictions:
  \begin{align*}
    v &\geq a-1 + \frac{4e}{a} \\
    v &\geq 2a-1 \\
    v &\geq a+3 \\
    v &\geq 2a+2 - \frac{4e}{a} + \sqrt{8a+1-\frac{32e}{a}}.
  \end{align*}
Consideration of the different regions of $\{(a,e) \in (\mathbb{R^{+}})^2\}$ where each of these is the most restrictive (largest) yields the bounds on $v$ given in \cref{eqn_v_bounds}.

\cref{chi_K_p} is always strictly less than \cref{puu_2}, and \cref{puu_1} is always less than or equal to \cref{puu_2}. So neither of these pairs give constraints on $v$.

The pairs \cref{chi_K_p} and \cref{pau}, and \cref{puu_1} and \cref{pau} both give
\[
  v \geq a-1 + \frac{4e}{a}.
\]

The constraint on $v$ derived from \cref{auu} being no greater than \cref{pau} implies 
\begin{equation}\label{4_and_7}
v \geq 2a-1.
\end{equation}
With $v \geq 2a-1$, \cref{auu} is less than \cref{puu_2} and no constraint on $v$ results. 
From $u \geq 0$ and \cref{pau} we get that
\[
  v \geq a + \frac{2e}{a}.
\]
But since $e \leq a(a-1)/2$, no constraint on $v$ is obtained when $v \geq 2a-1$.

Setting \cref{auu} less than or equal to \cref{u_binom_p_2} and isolating the square root term gives
\[
  \sqrt{8(v-a)+1} \leq v-2a+2 + \frac{4e}{a},
\]
where both sides are positive since $v \geq 2a-1$ by \cref{4_and_7}.  Since both sides are positive, an equivalent inequality is obtained by squaring both sides, giving 
\[
  \left(v-2a-2+ \frac{4e}{a}\right)^2 \geq 8a + 1 - \frac{32e}{a},
\]
which then implies
\begin{equation} \label{4_and_8}
  v \geq 2a + 2 - \frac{4e}{a} + \sqrt{8a + 1 - \frac{32 e}{a}}.
\end{equation}

From \cref{chi_K_p,u_binom_p_2} we have that
\[
  \frac{v-a-1}{2} \leq u \leq  v-a+\frac{1-\sqrt{8(v-a)+1}}{2},
\]
which gives $v \geq a+3$.

From $u \geq 0$ and \cref{u_binom_p_2}, we get that $v \geq a+1$, which gives no additional constraint since $v \geq a+3$. When $v \geq a+3$ the lower bound $u \geq 0$ is never larger than the upper bound from \cref{puu_2}. 

We have derived the constraints on $v$ from all pairs of lower and upper bounds on $u$, except for \cref{puu_1,u_binom_p_2}. \cref{tab: summary table} gives a summary of the constraints on $v$ presented in lexicographic order by equation numbers.

We now show that when $v \geq 2a-1$, then the constraint on $v$ derived from \cref{puu_1,u_binom_p_2} is less tight than \cref{4_and_8}, except for a finite number of values.

When $v=2a-1$, then \cref{puu_1} is smaller than \cref{u_binom_p_2} exactly when 
\[
  e \geq \frac{8a-6-2\sqrt{8a-7}}{4}.
\]
We use this value of $e$ to break the problem into cases.  As we proceed we will encounter finite sets of $(a,e,v)$ which are not covered by our computations.  We will collect these as we go and deal with them separately in Case 3.\\

\noindent\underline{\bf Case 1: $e < (8a-6-2\sqrt{8a-7})/4$}\\
\noindent
When $e = 0$, we get $v \geq 2a+2+\sqrt{8a+1}$. 

\cref{u_binom_p_2} does not depend on $e$ and  \cref{puu_1} decreases  as $e$ increases.  Thus if \cref{puu_1,u_binom_p_2} admit a non-empty range for $u$ (\ie $v$ is not forbidden by them) for a given $a$, $e$, and $v$, then they admit a positive range for all higher $e$.  Thus \cref{puu_1,u_binom_p_2} can only restrict $v < 2a+2+\sqrt{8a+1}$ when $e < (8a-6-2\sqrt{8a-7})/{4}$.\\
If in addition $a < 26$, then there are only a finite number of triples $(a,e,v)$.  We leave this finite set to Case 3. So assume that $a \geq 26$ and we obtain the following sequence of inequalities:

\begin{align*}
  3a-9 &\geq 2a + 2 + \sqrt{8a+1}\\
  3a-9 &\geq v \\
  3a-v-1 &\geq 8 \\
  \frac{e(3a-v-1)}{a} &\geq \frac{4e^2}{a^2} \\
  e &\geq \frac{e(v-2a+1)}{a} + \frac{4e^2}{a^2}\\
  \frac{(v-2a+1)^2 + 16e}{16} &\geq \frac{\left ( (v-2a+1) + \frac{8e}{a} \right)^2}{16} \\
  \frac{\sqrt{(v-2a+1)^2 + 16e}}{4} &\geq  \frac{(v-2a+1) + \frac{8e}{a}}{4} \\
 \frac{ -(v-2a+1) - \frac{8e}{a}}{4} &\geq \frac{-\sqrt{(v-2a+1)^2 + 16e}}{4} \\
  \frac{2v  -2 - \frac{8e}{a}}{4} &\geq \frac{3v - 2a - 1 -\sqrt{(v-2a+1)^2 + 16e}}{4}\\
  \frac{a(v  -1) - 4e}{2a} &\geq \frac{3v - 2a - 1 -\sqrt{(v-2a+1)^2 + 16e}}{4}. \\
\end{align*}
Thus in this range of $a$, $e$ and $v$, if \cref{auu} is less than \cref{u_binom_p_2} ($v$ is not ruled out by these two equations from \cref{u_bounds}), then  \cref{puu_1} is also less than \cref{u_binom_p_2}.  Put in another way: for $a \geq 26$, when \cref{puu_1} and \cref{u_binom_p_2} give a bound on $v$, then the bound derived from \cref{auu} and \cref{u_binom_p_2} is tighter.\\

\noindent\underline{{\bf Case 2:} $e \geq (8a-6-2\sqrt{8a-7})/4$}\\
We subdivide this case by the magnitude of $a$. If $a=1$, implying $e = 0$, then the two bounds from \cref{puu_1} and \cref{u_binom_p_2}  give $v \geq 7$. If $a=2$ and $e=0$, then the two bounds give $v \geq 11$.  When $a=3$ and $e=0$, then the two bounds give $v \geq 13$.  Since \cref{puu_1} has a negative derivative with respect to $e$, this means that for any larger $e$ the bounds are positively separated for $a=2$ and $v \geq 11$ or $a=3$ and $v \geq 13$. This leaves only a finite number of cases to check when $a = 2,3$ and $e >0$, which are left to Case~3. For $a\geq 4$ we proceed as follows: The partial derivative of \cref{puu_1} with respect to $v$ is
\[
  \frac{3}{4} + \frac{-(v-2a+1)}{4\sqrt{(v-2a+1)^2 + 16e}},
\]
while the partial derivative of \cref{u_binom_p_2} with respect to $v$ is
\[
  1 - \frac{2}{\sqrt{8(v-a)+1}},
\]
and thus the upper bound grows faster than the lower bound in $v$ whenever 
\begin{multline}\label{UB_grows_faster}
  8 \sqrt{(v-2a+1)^2+16e} \\
  <\sqrt{8(v-a)+1}\sqrt{(v-2a+1)^2+16e} + (v-2a+1)\sqrt{8(v-a)+1}.
\end{multline}
We first split into the cases $4 \leq a < 9$ and $a \geq 9$.\\

\noindent\underline{{\bf Case 2A:} $e \geq (8a-6-2\sqrt{8a-7})/4$ and $4 \leq a < 9$}\\
For $v < 4a$ we leave this finite set of values to Case 3. When $a \geq 4$ and $v \geq 4a$, then $\sqrt{8(v-a)+1} > 8$, \cref{UB_grows_faster} is true, and thus the derivative of \cref{u_binom_p_2} with respect to $v$ is larger than the derivative of \cref{puu_1}.  When $a \geq 4$ and $v=4a$, then \cref{puu_1} is smaller than \cref{u_binom_p_2} for any $e \geq 0$. Thus for all $a \geq 4$, $v \geq 4a$, and $e \geq 0$, \cref{puu_1} is always smaller than \cref{u_binom_p_2} and no additional constraint on $v$ is derived. \\

\noindent\underline{{\bf Case 2B:} $e \geq (8a-6-2\sqrt{8a-7})/4$ and $a \geq 9$}\\
When $a \geq 9$ and $v \geq 2a-1$, then $\sqrt{8(v-a)+1} > 8$, so \cref{UB_grows_faster} is true and the derivative of \cref{u_binom_p_2} with respect to $v$ is larger than the derivative of \cref{puu_1}. When $v = 2a-1$ and
\[
  e \geq \frac{8a-6-2\sqrt{8a-7}}{4},
\]
then \cref{puu_1} is smaller than \cref{u_binom_p_2}.  Thus for all $a\geq 9$, $v \geq 2a-1$, and $e \geq (8a-6-2\sqrt{8a-7})/4$, \cref{puu_1} is always smaller than \cref{u_binom_p_2} and no additional constraint on $v$ is derived.\\

\noindent\underline{{\bf Case 3:} Collected finite set of values $(a,e,v)$}\\
\noindent
The calculations in Cases~1 and~2 did not resolve the constraints from \cref{puu_1,u_binom_p_2} for the following ranges of $(a,e,v)$:
\begin{align*}
  a &<26 & e&<(8a-6-2\sqrt{8a-7})/4 & v&<2a +2 +\sqrt{8a+1} \\
  a &=2  & e&>0 & v&<11 \\
  a &=3  & e&>0 & v&<13 \\
  a &<9  & e&\geq(8a-6-2\sqrt{8a-7})/4 & v&<4a
\end{align*}
Directly calculating \cref{auu,puu_1,u_binom_p_2} for these finite sets, we find $60$ triples $(a,e,v)$ which are ruled out by \cref{puu_1,u_binom_p_2}, but permitted by \cref{auu,u_binom_p_2}. These are shown in \cref{5_8_not_4_8_all_a}.

\begin{table}[tb!]
  \centering
  \vspace{-2cm}
    \begin{tabular}{|c|c|c|}
    \hline
    $a$ & $e$ & $v$ \\
    \hline
    2   & 1 & 5--8\\
        \hline
    3   & 1 & 11,12\\
       & 2 & 8--10\\
       & 3 & 7,8\\
        \hline
    4   & 1 & 14\\
       & 2 & 13\\
       & 3 & 10--12\\
       & 4 & 8--11\\
        \hline
    5   & 2 & 16 \\
       & 3 & 15,16 \\
       & 4 & 13,14 \\
       & 5 & 11--13 \\
\hline
  \multicolumn{3}{c}{}\\
    \multicolumn{3}{c}{}\\
    \multicolumn{3}{c}{}\\
    \end{tabular}    \hspace{2cm}
    \begin{tabular}{|c|c|c|}
    \hline
    $a$ & $e$ & $v$ \\
    \hline
    6   & 1 & 20 \\
       & 2 & 19 \\
       & 3 & 18 \\
       & 4 & 17,18 \\
       & 5 & 16,17 \\
       & 6 & 15 \\
       & 7 & 13,14 \\ 
      \hline
    7   & 2 & 22 \\
       & 3 & 21 \\
       & 4 & 20 \\
       & 5 & 19,20 \\
       & 6 & 19 \\
       & 7 & 17,18 \\
       & 8 & 16 \\
\hline
      \multicolumn{3}{c}{}\\
    \end{tabular}    \hspace{2cm}
    \begin{tabular}{|c|c|c|}
  \hline
    $a$ & $e$ & $v$ \\
    \hline
    8   & 3 & 24 \\
       & 4 & 23 \\
       & 6 & 22 \\
       & 7 & 21 \\
       & 8 & 20 \\
       & 9 & 19 \\
    \hline
    9   & 1 & 28 \\
       & 4 & 26 \\
       & 6 & 25 \\
       & 7 & 24 \\
    \hline
    10   & 2 & 30 \\
       & 7 & 27 \\
    \hline
    12   & 4 & 34 \\ 
      \hline
    \multicolumn{3}{c}{}\\
    \multicolumn{3}{c}{}\\
    \end{tabular}
  \caption{Values of $a$, $e$ and $v$ forbidden by \cref{puu_1,u_binom_p_2} but permitted by \cref{4_and_8} for $v \geq 2a-1$.\label{5_8_not_4_8_all_a}}
\end{table}

\renewcommand{\arraystretch}{1.5}
\begin{table}[tb!]
\begin{center}
\begin{tabular}{|c|c|c|}\hline
Lower Bound on $u$&Upper Bound on $u$&Constraints on $v$ \\\hline
0&~\eqref{puu_2}& None (given $v \geq a+3$)\\\hline
  0&~\eqref{pau}&None (given $v \geq 2a-1$)\\\hline
0&~\eqref{u_binom_p_2}&None (given $v \geq a+3$)\\\hline
~\eqref{chi_K_p}&~\eqref{puu_2}&None\\\hline
~\eqref{chi_K_p}&~\eqref{pau}&$v \geq a-1+\frac{4e}{a}$\\\hline
~\eqref{chi_K_p}&~\eqref{u_binom_p_2}&$v \geq a+3$\\\hline
~\eqref{auu}&~\eqref{puu_2}&None (given $v \geq 2a-1$)\\\hline
~\eqref{auu}&~\eqref{pau}& $v \geq 2a-1$\\\hline
~\eqref{auu}&~\eqref{u_binom_p_2}&$v \geq   2a + 2 - \frac{4e}{a} + \sqrt{8a + 1 - \frac{32 e}{a}}$ \\
&&(given $v \geq 2a-1$)\\\hline
~\eqref{puu_1}&~\eqref{puu_2}&None\\\hline
~\eqref{puu_1}&~\eqref{pau}&$v \geq a-1+\frac{4e}{a}$\\\hline
  ~\eqref{puu_1}&~\eqref{u_binom_p_2}& Finite number of forbidden values \\
  && See \cref{5_8_not_4_8_all_a} \\\hline
\end{tabular}
\caption{Summary of constraints on $v$ given upper and lower bounds on $u$.}\label{tab: summary table}
\end{center}
\end{table}
\renewcommand{\arraystretch}{1}

We have now considered all pairs of lower and upper bounds.  The constraints on $v$ that they imply are shown in \cref{tab: summary table}. Except for the finite number of triples $(a,e,v)$ shown in \cref{5_8_not_4_8_all_a}, the bounds that are not always exceeded by some others are
\begin{align*}
  v &\geq a -1 + \frac{4e}{a}, \\
  v &\geq 2a-1, \\
  v &\geq 2a + 2 - \frac{4e}{a} + \sqrt{8a+1 -\frac{32e}{a}}, \\
  v &\geq a+3.
\end{align*}
The ranges of values $(a,e)$ for which each of these bounds is the tightest are found by comparing them in pairs. These are summarized in \cref{eqn_v_bounds}. None of the values in \cref{5_8_not_4_8_all_a} are ruled out by any of these four bounds on $v$, thus they are the full set of exceptional values forbidden by \cref{puu_1,u_binom_p_2}.

\end{proof}
We note that we have not taken into account the fact that the total number of points $v$ is restricted as we are playing on an $\STS$, i.e.\ $v\equiv 1,3 \pmod 6$.  For example, when taking these constraints into consideration, then $v \geq 2a-1$ becomes
\[
  v \geq \begin{cases} 2a+1 & a \equiv 0 \pmod 3 \\ 2a-1 & a \equiv 1,2 \pmod 3 \end{cases}
\]
and $v \geq a+3$ becomes
\[
  v \geq \begin{cases} a+3 & a \equiv 1,3 \pmod 6 \\ a+6 & a \equiv 1 \pmod 6 \\ a+5 & a \equiv 2 \pmod 6 \\ a+4 & a \equiv 3,5 \pmod 6 \end{cases}.
  \]
The other two bounds do not have simple closed form expressions because they depend on the ratio of $e$ and $a$ and one contains a square root.

We have also not taken into account that the number of unplayable vertices, $u$, must be an integer.  For example $(a,e,v) = (7,7,19)$ is permitted by \cref{v_bounds}, but cannot actually happen even though the lower bound on $u$ from \cref{auu}, $21/2 -\sqrt{37}/2$, is strictly smaller than the upper bound on $u$ from \cref{u_binom_p_2}, $ 25/2  -\sqrt{97}/2$.  In this case, since $u$ is a non-negative integer, these bounds give $8 \leq u \leq 7$, which is a contradiction.  Finally the bounds on $v$ in \cref{v_bounds} also ignore \cref{chi_g} and \cref{chi_comp_g} from \cref{u_bounds}, but they have the advantage of being simple to compute if only $a$ and $e$ are known.

Taking into account that $v \equiv 1,3 \pmod 6$ and that $u$ is a non-negative integer changes the set of exceptional triples $(a,e,v)$ to those shown in \cref{reduced_list_of_exceptions}.  Note that this set is not a subset of the triples from \cref{5_8_not_4_8_all_a} because the restriction that $u$ is a non-negative integer tightens both the bound from \cref{puu_1} and \cref{u_binom_p_2} and the bound from \cref{auu} and \cref{u_binom_p_2}.
\begin{table}[tb!]
\centering
  \begin{tabular}{|c|c|c|}
    \hline
    $a$ & $e$ & $v$ \\
    \hline
    2   & 1 & 7,9\\
        \hline
    3   & 2 & 9\\
       & 3 & 7\\
        \hline
    4   & 2 & 13,15\\
       & 3 & 13,15\\
       & 4 & 9,15\\
       & 5 & 9\\
        \hline
    5   & 3 & 15 \\
       & 5 & 13 \\
       & 6 & 13 \\
       & 7 & 13 \\
    \hline
      \multicolumn{3}{c}{}\\
      \multicolumn{3}{c}{}\\
      \multicolumn{3}{c}{}\\
      \multicolumn{3}{c}{}\\
\end{tabular}\hspace{2cm}
    \begin{tabular}{|c|c|c|}
    \hline
      $a$ & $e$ & $v$ \\
\hline
        6   & 3 & 19 \\
       & 4 & 19 \\
       & 6 & 15 \\
       & 7 & 15 \\ 
       & 9 & 13 \\
       & 10 & 13 \\ 
       & 11 & 13 \\
    \hline
    7   & 4 & 21 \\
       & 7 & 19 \\
       & 8 & 19 \\
       & 9 & 19 \\
       & 11 & 15 \\
\hline
      \multicolumn{3}{c}{}\\
      \multicolumn{3}{c}{}\\
      \multicolumn{3}{c}{}\\
\end{tabular}\hspace{2cm}
    \begin{tabular}{|c|c|c|}
    \hline
      $a$ & $e$ & $v$ \\
    \hline
    8   & 4 & 25 \\
       & 5 & 25 \\
       & 8 & 21 \\
       & 9 & 21 \\
       & 12 & 19 \\
       & 13 & 19 \\
       & 14 & 19 \\
    \hline
    9   & 5 & 27 \\
       & 9 & 25 \\
       & 10 & 25 \\
       & 11 & 25 \\
       & 14 & 21 \\
    \hline
    10   & 10 & 27 \\
       & 11 & 27 \\
    \hline
    11   & 17 & 27 \\
    \hline
  \end{tabular}
  \caption{Values of $a$, $e$ and $v$ forbidden by \cref{puu_1} and \cref{u_binom_p_2} but permitted by \cref{4_and_8} for $v \geq 2a-1$ taking into account that $v \equiv 1,3 \pmod 6$ and $u$ is integral.\label{reduced_list_of_exceptions}}
\end{table}

\cref{v_bounds} allows us to understand something we encountered in \cref{sec_nim values}: why we could not have encountered the graph $P_3 \dot\cup K_1$ in any of the Steiner triple systems of orders 15 or smaller.  This graph has $a=4$ and $e=2$.  For $a=4$, $e = 2 < (-a+\sqrt{8a^3-7a^2})/4$.  \cref{v_bounds} shows that $v \geq 2a + 2 - 4e/a + \sqrt{8a+1-32e/a} = 8 + \sqrt{17} \approx 12.12$ which shows that $v \geq 13$ is a necessary condition.  But $v = 13$ is eliminated by \cref{puu_1,u_binom_p_2} as shown in \cref{5_8_not_4_8_all_a}.  If this graph could be embedded in an $\STS(15)$ we turn to \cref{u_bounds} and find that \cref{puu_1} shows that $u \geq 6.55$ and \cref{u_binom_p_2} shows that $u \leq 6.78$.  These two bounds can be satisfied for real valued $u$ but they are impossible because $u$ is a non-negative integer.  Thus an $\STS(19)$ is the smallest Steiner triple system in which this graph can be embedded and our experiments found an instance of this graph embedded in an $\STS(19)$.


\section{Sufficient conditions for embedding a graph in an $\STS$}\label{sufficient}

\cref{sec_restrictions} gives restrictions on the sizes of $P$, $A$, and $U$ and required structures of some of the blocks in any embedding of a given graph in an $\STS$. These are necessary conditions for a graph to be embeddable in an $\STS$.  In this section we determine a sufficient condition for a graph to be embeddable in an $\STS$. In particular we will show that any graph can be embedded in a sufficiently large $\STS$. 

The general strategy for embedding a given graph $G$ in a large $\STS$ will be to choose sizes for $P$ and $U$ that satisfy the necessary conditions.  Then we pick triples with the properties that are required: all edges of $G$ appear in $PAA$, all non-edges of $G$ appear in $AAU$, all edges with $P$ are in $PPU$, and every vertex of $U$ appears in at least one block from $PPU$. From the complete graph $K_{P\cup A \cup U}$ we remove all edges that appear in any of these chosen triples.  If we can decompose the graph that remains into triples, we will have successfully embedded $G$ in an $\STS$.

A graph $G$ is {\em $C_3$-divisible} if the number of edges is a multiple of 3 and the degree of every vertex is even.  Because each triple uses three edges and contributes 2 to the degree of its points, being $C_3$-divisible is a necessary condition for a graph to be decomposable into triples. Barber \etal have proven that any sufficiently large graph of  high minimum degree which is $C_3$-divisible can be decomposed into triangles \cite{MR3436388}.

\begin{theorem}[Theorem 1.4(ii) $l=3$, \cite{MR3436388}] \label{barber_thm}
  For each $\epsilon > 0$ there exists an $n_0 = n_0(\epsilon,3)$ such that every $C_3$-divisible graph $G$ on $n \geq n_0$ vertices with $\delta(G) \geq (9/10 + \epsilon)n$ has a $C_3$-decomposition.
  \end{theorem}
  
We use \cref{barber_thm} to show that every graph can be embedded in a Steiner triple system for all sufficiently large orders.
\begin{theorem}\label{thm_asymp}
  Let $G$ be a graph.  There exists a $v_G$ such that for all $v \geq v_G$ with $v \equiv 1,3 \pmod 6$ an $\STS(v)$ exists with $G$ as a follower.
  \end{theorem}
  \begin{proof}
    We will start by defining $v_G$, which is chosen to be large enough that we can both define some edge-colourings which we use to build blocks and satisfy the conditions of \cref{barber_thm}.  For any $v \geq v_G$ we will choose $p=|P|$ and $u=|U|$ so that all the necessary conditions from \cref{u_bounds} are satisfied.  We will construct all the blocks of $PAA$, $AAU$, and $PPU$ as required. Removing the edges from all these triples from $K_v$ will leave a $C_3$-divisible graph which is large enough and has degree high enough to use \cref{barber_thm}.
    
    Let $a = |V(G)|$, $e = |E(G)|$, and $b = \binom{a}{2} - e$, the number of edges in the complementary graph, $\overline{G}$.  Let $\epsilon >0$ and $n_0 = n_0(\epsilon,3)$ be the constant from \cref{barber_thm}. Let $c = (1/10 -\epsilon)$. Let
    \[
      v_G = \max \left (\frac{e(e +1)}{2} + a, a + b + 2 + \sqrt{2b+3}, 3\left (1+\frac{2}{c} \right )^2+a, \frac{2a-1}{c}, n_0\right).
    \]
Briefly, we will need $v \geq \frac{e(e +1)}{2} + a$ to construct an injective edge colouring of $G$ with no more than $p$ colours.  We will need $v \geq a + b + 2 + \sqrt{2b+3}$ to construct an injective edge colouring of $\overline{G}$ with no more than $u$ colours. We will use that $v \geq 3\left (1+\frac{2}{c} \right )^2+a$ and $v\geq \frac{2a-1}{c}$ to ensure two conditions.  The first is to build a surjective edge colouring of $K_P$ with exactly $u$ colours.  The second is to satisfy the degree condition from \cref{barber_thm}. Finally $v \geq n_0$ ensures that the graph is large enough to use \cref{barber_thm}.
    
Suppose that $v \geq v_G$ and $v \equiv 1,3 \pmod 6$. Set $m = v-a$, $p = \left \lceil \sqrt{2m} \right \rceil$, and $u=m-p$. Note that $\sqrt{2m} \leq p < \sqrt{2m} + 1$.  Let $P$ and $U$ be sets of $p$ and $u$ points respectively.

As in the proof of \cref{u_bounds}, the blocks of $PAA$ will determine a proper edge colouring of $G$ with $p$ colours and we will need this colouring to build the triples of $PAA$.  Since $v \geq \binom{e+1}{2} + a$, we have that $m \geq \binom{e+1}{2}$ and thus
    \begin{equation}
      p \geq e. \label{chi_G}
    \end{equation}
    Thus an injective  proper edge colouring $\phi: E(G) \rightarrow P$ exists. The injectivity will help us establish the degree condition.
    
    Similarly, determining the blocks of $AAU$ is equivalent to having a proper edge colouring of $\overline{G}$. Since $v \geq a + b+2 + \sqrt{2b+3}$, we have that $m -(b+2) \geq \sqrt{2b+3}$.  This implies that $(m-(b+2))^2 \geq 2b+3$ or, equivalently, $m^2 + (-2b-4)m + (b+1)^2 \geq 0$.  From this we have $m^2 +(-2b-2)m + (b+1)^2 \geq 2m$ and thus $(m-(b+1))^2\geq 2m$.  Taking the square root of both sides we get $m -(b+1) \geq \sqrt{2m} > p-1$. In turn this means that
    \begin{equation}
      u > b, \label{chi_G_comp}
    \end{equation}
    which guarantees that an injective proper edge colouring $f: E(\overline{G}) \rightarrow U$ exists. Injectivity is again used to prove the needed degree condition from \cref{barber_thm}.

    Finally, for the $PPU$ triples we need a proper edge colouring of $K_{P}$ with $u$ colours.  But since every point in $U$ must appear in a triple of $PPU$, every colour class from the edge colouring must contain at least one edge. First note that $c < 1/10$ means that $1/c > 10$ and thus $v \geq 3(1+2/c)^2+ a > 1323 + a$.  So $m = v-a > 1323 > 32$.  If $m \geq 32$, then $m \geq 4 \sqrt{2m}$ and $m \geq (3m + 4\sqrt{2m})/4$.  Now $m \geq 4 \sqrt{2m}$ and $m \geq 32$ give
    \begin{align*}
      4 + \sqrt{2m} &\leq m\\
      2m + 4 + 5 \sqrt{2m} &\leq 3m + 4\sqrt{2m}\\
      (\sqrt{2m}+1)(\sqrt{2m}+4) &\leq 3m + 4\sqrt{2m}\\
      p(p+3) &\leq 3m + 4\sqrt{2m}.
    \end{align*}
    This, together with $m \geq (3m + 4\sqrt{2m})/4$, gives that $m \geq (p^2+3p)/4$, which implies that
      \begin{equation}
        u > \frac{\binom{p}{2}}{2}. \label{chi_K_P_lb}
      \end{equation}
      Finally,
    \begin{align*}
      p+ \binom{p}{2} &\geq \sqrt{2m} + \frac{\sqrt{2m}(\sqrt{2m}-1)}{2} \\
                      &= \frac{\sqrt{2m}(\sqrt{2m}+1)}{2} \\
      &\geq m,
    \end{align*}
    so that
    \begin{equation}
      u < \binom{p}{2}. \label{chi_K_P_ub}
      \end{equation}
      The chromatic index of $K_{P}$ is $p$ or $p-1$, depending on whether $p$ is odd or even.  Since $m \geq 32$, we have $p \geq \sqrt{2m} \geq \sqrt{64} = 8$ and thus $u \geq \binom{p}{2}/2 \geq p$ and a proper edge colouring of $K_P$ with $u$ colours exists. Note that this colouring may have empty colour classes. If the size of two colour classes differ by more than one, then the union of the two colour classes consists only of even cycles and paths and at least one path has odd length and begins and ends with an edge from the larger colour class.  Switching the edge colours on this path reduces the difference between the sizes of the colour classes.  By repeating this procedure we can obtain an equitable edge colouring where the sizes of any two colour classes differ by at most one \cite{west_introduction_2018}.  Since
      \[
        u \geq \frac{\binom{p}{2}}{2} \geq p,
      \]
      the largest colour class has size 2 and the smallest colour class is at least one.  Thus a map $g: E(K_{P}) \rightarrow U$ which is a surjective proper edge colouring exists.

      We now build the $\STS$ on point set $P \cup A \cup U$ where $A = V(G)$.  
    From the three proper edge colourings on $G$, $\overline{G}$, and $K_P$, define three sets of triples
    \begin{align*}
      PAA &= \{\{x,y,\phi(\{x,y\})\}: \{x,y\} \in E(G)\} \\
      AAU &= \{\{x,y,f(\{x,y\})\}: \{x,y\} \in E(\overline{G})\} \\
      PPU &= \{\{x,y,g(\{x,y\})\}: x,y \in P \}.
      \end{align*}
      The set $PAA$ only uses edges between $P$ and $A$ and all the edges of $G$. The set $AAU$ only uses edges between $U$ and $A$ and all the edges of $\overline{G}$. And the set $PPU$ only uses edges between $P$ and $U$ and all the edges within $P$.  Since $\phi$, $f$, and $g$ are proper edge colourings, all these triples are edge-disjoint.  Since $g$ is surjective, every point in $U$ is on a triple from $PPU$. We remove the edges of all these triples from $K_{P \cup A \cup U} = K_{v}$.  Since $v \equiv 1,3 \pmod 6$, if we remove a set of edge-disjoint triples from $K_v$, then the resulting graph is $C_3$-divisible. Thus the graph remaining on vertex set $P \cup A \cup U$ after removing the edges of these triples is $C_3$-divisible.  Since $v \geq n_0$, to invoke \cref{barber_thm}, we now only need to show that the degree of each vertex is at least $(9/10 + \epsilon)v$, or equivalently that the degree in $PPA \cup AAU \cup PPU$ is at most $cv-1$.

      For $x \in P$, $\deg(x) = 2(|\phi^{-1}(x)| + p -1) \leq 2p$ since $\phi$ is injective.  Because $v \geq 3(1+2/c)^2+a$, we have that $m \geq \sqrt{3/2} \sqrt{2m} (1+2/c)$, which gives $cu \geq 2p$.   Thus $cv -1 \geq cu \geq 2p$. For $x \in U$, $\deg(x) = 2(|f^{-1}(x)| + |g^{-1}(x)|) \leq 6$ since $f$ is injective and $g$ is equitable with $u \geq \binom{p}{2}/2$. Since $v \geq 3(1+2/c)^2 \geq 7/c$, we have $\deg(x) \leq cv -1$.  Finally for $x \in A$, $\deg(x) = 2(a-1)$ and the fact that $v \geq (2a-1)/c$ shows that the degree is small enough.  Thus \cref{barber_thm} holds and there exists an $\STS$ on $P \cup A \cup U$ containing the triples in $PAA$, $AAU$, and $PPU$.

      The position where exactly the points in $P$ have been played is a follower of the game on this $\STS$.  Because $g$ is surjective, every point $x \in U$ is on a block with two points from $P$ and therefore each point in $U$ cannot be played. Every pair of points  $x, y \in P$ is in a block $\{x,y,g(\{x,y\})\}$ with $g(\{x,y\}) \in U$, thus no points of $A$ are in a block with two points from $P$ and every point in $A$ can be played.    The set of triples $PAA$ forbids the vertices of any edge of $G$ from both being played.  The set of triples $AAU$ ensures that the only events preventing two non-adjacent vertices in $G$ from both being played are plays made in this position or after.  Thus the game tree of this position is isomorphic to the game tree of \nodekayles played on $G$.
\end{proof}

Every graph $G$ can be embedded in an $\STS$ such that \nofil on this $\STS$ contains a follower equivalent to \nodekayles played on $G$.  Since $n_0$ and $c$ are constants not depending on $G$, and the rest of the equations bounding $v_G$ from below are polynomial in the size of $G$, the size of the smallest $\STS$ containing an embedding of $G$ is polynomial in $G$.  Since determining the outcome of \nodekayles played on graphs is PSPACE-complete, we have the following.
\begin{corollary}\label{complexity}
  The complexity of determining the outcome class of \nofil on Steiner triple systems and the outcome classes of the followers of these games is PSPACE-complete.
  \end{corollary}

\section{Conclusion}

We have begun the exploration of the game \nofil played on Steiner triple systems.  For orders up to 15, by computing the nim-values and exploring the game trees, we have determined the outcome of the game when played using an optimal strategy.  For orders 19, 21, and 25 we have determined nim-values for a significant number of sampled Steiner triple systems.  Open questions include:  What is the set of nim-values encountered at the start of playing \nofil on all Steiner triple systems?  Is this set infinite? Do all possible nim-values occur? Do the nim-values for designs with larger block sizes behave differently than $\STS$s?  The experiments on $\STS$s prompt some more specific questions.  How strong is the apparent pattern that odd nim-values predominate when $v \equiv 1 \pmod 6$ and even nim-values predominate when $v \equiv 3 \pmod 6$?  

Counting the allocation of played, available, and unplayable points in the blocks of an $\STS$, we derived necessary conditions for a graph $G$ to be embedded in an $\STS$.  We showed that a weaker set of conditions is sufficient for $G$ to be embedded.  The embeddings constructed in \cref{thm_asymp} are in Steiner triple systems which may be much larger than the bounds from \cref{v_bounds} require.  Some initial experiments with various small graphs seem to indicate that the bounds from \cref{u_bounds} and \cref{v_bounds} are tight or very close to being tight.  The current focus of our research is determining minimal embeddings for some natural families of graphs.

Finally, it is possible that it could be proven that the optimal lines of play can always avoid the graphs that determine that \nodekayles is PSPACE-complete.  If this were true, then the PSPACE status of the initial games on Steiner triple systems would be undetermined.  Therefore stronger methods of analyzing the complexity of \nofil on Steiner triple systems and designs in general is of significant interest.

\bibliographystyle{plain}
\bibliography{refs}

\end{document}